\documentclass[11pt]{elsarticle}

\usepackage{amsfonts}
\usepackage{ae} 
\usepackage[T1]{fontenc}
\usepackage[ansinew]{inputenc}
\usepackage{amsmath,amssymb,amsthm}
\usepackage{graphicx}
\usepackage{bbding}
\usepackage{color}
\usepackage[colorlinks]{hyperref}
\usepackage{lscape}
\usepackage{epsfig}
\usepackage{mathrsfs}
\usepackage{dsfont}
\usepackage{bbm}
\usepackage{fancyhdr}
\usepackage{longtable}

\usepackage{lineno,hyperref}

\journal{arXiv}

\begin{document}

\newtheorem{corollary}{Corollary}[section]
\newtheorem{theorem}{Theorem}[section]
\newtheorem{lemma}{Lemma}[section]
\newtheorem{proposition}{Proposition}[section]
\theoremstyle{remark}
\newtheorem{remark}{Remark}[section]
\newtheorem{example}{Example}[section]
\theoremstyle{definition}
\newtheorem{definition}{Definition}[section]
\newtheorem{notation}{Notation}[section]
\numberwithin{equation}{section}

\renewcommand{\theequation}{\thesection.\arabic{equation}}
\renewcommand{\thefigure}{\thesection.\arabic{figure}}
\renewcommand{\thetable}{\thesection.\arabic{table}}
\renewcommand{\theequation}{\thesection.\arabic{equation}}

\providecommand{\norm}[1]{\lVert#1\rVert}
\providecommand{\abs}[1]{\lvert#1\rvert}
\newcommand{\G}{\Gamma}

\newcommand{\vect}[1]{\mathbf{#1}}
\renewcommand{\t}{\tau}
\renewcommand{\o}{\omega}
\newcommand{\s}{\sigma}
\newcommand{\var}{\varphi}
\newcommand{\e}{\varepsilon}
\newcommand{\al}{\alpha}
\newcommand{\bt}{\beta}
\newcommand{\ba}{\begin{array} }
\newcommand{\ea}{\end{array} }
\newcommand{\be}{\begin{equation} }
\newcommand{\ee}{\end{equation} }
\newcommand{\baa}{\begin{align} }
\newcommand{\eaa}{\end{align} }
\newcommand{\da}{\delta}
\newcommand{\kal}{\kappa}
\newcommand{\Da}{\Delta}
\renewcommand{\l}{\lambda}
\newcommand{\laa}{\Lambda}
\newcommand{\f}{\displaystyle\frac}
\newcommand{\il}{\displaystyle\int\limits}
\newcommand{\li}{\lim\limits}
\newcommand{\ld}{\displaystyle\left\{}
\newcommand{\rd}{\displaystyle\right\}}
\newcommand{\lx}{\displaystyle\left(}
\newcommand{\rx}{\displaystyle\right)}
\newcommand{\n}{\noindent}
\newcommand{\va}{\vartheta}
\newcommand{\ga}{\gamma}
\newcommand{\cL}{{\mathscr{L}}}
\newcommand{\ma}{{\mathcal{A}}}
\newcommand{\mab}{{\tilde{\mathcal{A}}}}
\newcommand{\ra}{\rightarrow}
\newcommand{\op}{\oplus}
\newcommand{\p}{\bar{P}}
\newcommand{\pt}{\tilde{P}}
\newcommand{\qtt}{\tilde{Q}}
\newcommand{\q}{\bar{Q}}
\newcommand{\T}{{\mathbb{T}}}
\newcommand{\J}{{\mathbb{J}}}
\newcommand{\C}{{\mathbb{C}}}
\renewcommand{\O}{\Omega}
\newcommand{\R}{{\mathbb{R}}}
\newcommand{\A}{{\mathbb{A}}}
\newcommand{\Z}{{\mathbb{Z}}}
\newcommand{\I}{{\mathbb{I}}}
\newcommand{\N}{{\mathbb{N}}}
\newcommand{\B}{{\mathcal{B}}}
\newcommand{\X}{{\mathcal{X}}}
\newcommand{\Y}{{\mathcal{Y}}}
\newcommand{\W}{{\mathcal{W}}}
\newcommand{\pa}{\partial}
\newcommand{\crd}{{\rm C}_{\rm rd}}
\newcommand{\cf}{{\rm C}}
\newcommand{\del}{^{\Delta}}
\newcommand{\tx}[1]{\quad\mbox{#1}\quad}
\newcommand{\ove}[1]{\overline{#1}}
\newcommand{\vecc}[2]{\begin{bmatrix}#1\\#2\end{bmatrix}}
\newcommand{\veccc}[2]{\begin{bmatrix}#1\\[.5cm]#2\end{bmatrix}}
\newcommand{\tu}[1]{\textup{#1}}
\newcommand{\td}{\tilde}
\newcommand{\ve}{\varepsilon}
\newcommand{\iy}{\infty}
\newcommand{\ta}{\theta}
\newcommand{\om}{\ominus}

\begin{frontmatter}

\title{Existence and Stability of Positive Periodic Solutions to Functional Differential Equations}
\tnotetext[mytitlenote]{Supported by NSFC-11271065, RFPD-20130043110001, and NSF grant DMS-0077790.}


\author[nenu]{
{Meng Fan}}
\ead{mfan@nenu.edu.cn}

\author[asu]{
{Yang Kuang}}
\ead{kuang@asu.edu}

\author[asuwest]{
{Haiyan Wang}\corref{mycorrespondingauthor}}
\cortext[mycorrespondingauthor]{Corresponding author}
\ead{haiyan.wang@asu.edu}

\author[nenu]{Shaojiang Yu}
\ead{yusj353@nenu.edu.cn}

\address[nenu]{School of Mathematics and Statistics, Northeast Normal University, 5268 Renmin Street, Changchun, Jilin, 130024, P. R. China.}
\address[asu]{School of Mathematics and Statistics, Arizona State University, Tempe, AZ 85287, USA}
\address[asuwest]{School of Mathematical and Natural Sciences, Arizona State University, Phoenix, Arizona 85069, USA}

\begin{abstract}
We establish several delay-independent criteria for the existence and stability of positive periodic solutions of $n$-dimensional nonautonomous functional differential equation by several fixed point theorems.  Examples from positive and negative growth feedback systems demonstrate the results.

\end{abstract}


\end{frontmatter}


\section{Introduction}
The equation  of the form
\begin{equation}\label{eq2}
x'(t)=-a(t)x(t) + \l b(t)f(x(t-\tau(t)))
\end{equation}
and its generalizations have been proposed as models for a variety of population dynamics and
physiological processes such as production of blood cells
\cite{WAZLASTA} and other applications \cite{GURNAYBN,MACKEYGLASS}.   The existence and stability of periodic solutions of the equation of form \eqref{eq2} have attracted much attention. See, e.g., \cite{BI,CHENGZHANG,CHOW1,FREEDMANWU1992,KUANGBOOK,KUANGJAPAN,KUANGSMITH1,MALLETPARETNUSSBAUM,TANGKUANG1997,HYW,YE}.

Inspired by multiple-species ecological models, it is natural to
explore $n$-dimensional systems. Autonomous systems have been extensively studied. However, natural environment is physically highly variable and many processes are usually subject to seasonal fluctuations and vary greatly in time, which is a more practical scenario. If a model is desired which takes into account such temporal inhomogeneity, it must be nonautonomous, which is, of course, more difficult to study in general.
One must ascribe some properties to the time dependence of parameters in the models, for only then can the resulting dynamic be studied accordingly.
 One of the methods of incorporating temporal nonuniformity of the
environments in models is to assume that the parameters are
periodic with the same period of the time variable.

Wang \cite{HYW} considered the existence, multiplicity and nonexistence of positive
$\omega$-periodic solutions for the periodic scalar equation
$$x'(t)=a(t)g(x(t))x(t)-\lambda b(t)f(x(t-\tau(t))),$$
and show that the number of positive $\omega$-periodic solutions can be determined by indexes.
In this paper, we will extend the results to the $n$-dimensional nonautonomous system $\dot{\vect{x}}(t)=-\vect{A}(t)\vect{x}(t) + \l \vect{B}(t)\vect{F}(\vect{x}(t-\t(t)))$.

On the other hand, stability of solution has attracted much attention \cite{CHOWMALLET2}, and
Liapunov's direct method has been the primary tool to deal with stability problems \cite{BURTONBOOK,KG,HALE,KUANGBOOK,TY}.  However, there remains numerous challenging problems in stability and new methods are needed to address the difficulties \cite{BURTONBOOK,XIA}.

Burton \cite{BURTON,BURTONBOOK} used fixed point theories to investigate stability of differential equations. Zhang studied the stability of certain types of delay equations in \cite{ZB1, ZB2, ZB3}, and the stability properties of the scalar delay equation $$x'(t)=-a(t)x(t)+g(t,x_t)$$ by means of fixed point theories.

Fan, Xia and Zhu \cite{XIA} considered the stability of the scalar delay differential equations
$$
x'(t)=-a(t,x_t)x(t)+f(t,x_t),~~x'(t)=-g(t,x(t))+f(t,x_t),
$$
by the contractive mapping principle and Schauder's fixed point theorem. In this paper we will establish sufficient and necessary criteria for the asymptotic stability of the $n$-dimensional nonautonomous system $\vect{x}'(t)=-\vect{A}(t)\vect{x}(t) + \vect{B}(t)\vect{F}(\vect{x}_t)$ with the contractive mapping principle approach.

The paper is organized as follows. We first study existence, multiplicity and nonexistence of periodic solutions of the $n$-dimensional nonautonomous system with delay.  In section \ref{stability sec}, we establish sufficient and necessary criteria for the stability of positive periodic solution for the $n$-dimensional nonautonomous system with delay.  All the proofs of the main findings are postponed to Appendix.  In Sections \ref{feedbacksyssec} and \ref{fdeexpsec} examples including positive and negative growth feedback systems  and their numerical simulations are given to illustrate the results. Finally we give the detailed proofs for the existence, multiplicity and nonexistence of positive periodic solutions and also the stability of the system in Appendix (\S \ref{appendixsec}).

\setcounter{equation}{0}
\section{Existence of Positive Periodic Solutions}
In this section, we shall study the existence, multiplicity and
nonexistence of positive $\o$-periodic solutions for the
nonautonomous $n$-dimensional delay system
\begin{equation}\label{eq1}
\vect{x}'(t)=-\vect{A}(t)\vect{x}(t) + \l \vect{B}(t)\vect{F}(\vect{x}(t-\t(t))),
\end{equation}
where $\vect{A}(t) = \mathop{\rm diag}[a_1(t),a_2(t),\dots,
a_n(t)]$, $\vect{B}(t) = \mathop{\rm diag}[b_1(t),b_2(t),\dots,
b_n(t)]$, $\vect{F}(\vect{x}) =[f^1(\vect{x}),f^2(\vect{x}),\dots,
f^n(\vect{x})]^{T}$, and $\l>0$ is a positive parameter.

Let $\mathbb{R}=(-\infty,
\infty)$, $\mathbb{R}_+=[0, \infty)$, $\mathbb{R}_+^n=\Pi_{i=1}^n
\mathbb{R}_+$ and, for any $\vect{x}=(x_1,...,x_n) \in
\mathbb{R}^n_+$, $\norm{\vect{x}}=\sum_{i=1}^n \abs{x_i}$.
For any $\omega$-periodic function $g$, denote its average by $\bar{g}=\f{1}{\omega}\il_0^\omega g(s)ds$.

For system (\ref{eq1}), we assume that
\begin{description}
\item{(H1)~~$a_i, b_i$ $\in C(\mathbb{R}, [0, \infty))$ are
$\o$-periodic such that $\bar{a}_i>0$, $\bar{b}_i>0$, $i=1,\dots,n$ and $\t \in C(\mathbb{R},\mathbb{R})$ is $\o$-periodic.}
\item{(H2)~~$f^i:
\mathbb{R}_+^n \to [0,\infty)$ is continuous with $ f^i(\vect{x})
> 0 $ for $\norm{\vect{x}} > 0,$ $i=1,\dots,n$,}
\end{description}

In order to state our theorems, we further introduce some notations. Let
$$
\begin{array}{l}
\sigma=\min\limits_{i=1,\dots,n}\{\sigma_i \},~~{\rm where}~~\sigma_i =
e^{-\bar{a}_i\omega}, i=1,\dots,n;\\
M(r)=\max\limits_{i=1,\dots,n}\{f^i(\vect{x}): \vect{x} \in
\mathbb{R}_+^n, \norm{\vect{x}} \leq r\};\\
m(r)=\min\limits_{i=1,\dots,n}\{f^i(\vect{x}): \vect{x} \in
\mathbb{R}_+^n, \s r\leq  \norm{\vect{x}} \leq r\};\\
\G=\min\limits_{i=1,\dots,n}\{(\sigma_i^{-1}-1)^{-1}\bar{b}_i\omega\},~~\chi = \displaystyle\sum_{i=1}^n
\f{\sigma^{-1}_i}{\sigma^{-1}_i-1}\bar{b}_i\omega;\\
\ f_0^i =  \li_{\norm{\vect{x}} \to 0} \f{f^i(\vect{x})}{\norm{\vect{x}}},~~f_{\infty}^i =\li_{\norm{\vect{x}} \to \infty} \f{f^i(\vect{x})}{\norm{\vect{x}}},~~i =1,...,n;\\
\vect{F}_0= \max_{i=1,\dots,n}\{f^i_0\},~~ \vect{F}_{\infty}= \max_{i=1,\dots,n}\{f^i_{\infty}\}.
\end{array}
$$
Moreover, define
$$
\ba{l}
 i_0=\text{number of zeros in the set } \{\vect{F}_0, \vect{F}_\infty\},\\
  i_\infty=\text{number of infinities in the set } \{\vect{F}_0,\vect{F}_\infty\}.
\ea
$$
It is clear that $i_0$ and $i_{\infty}$ can be $0,1,\text{or }2$. In
the following, we shall show that (\ref{eq1}) has $i_0$ or
$i_{\infty}$ positive $\o$-periodic solution(s) for sufficiently
large or small $\l$, respectively. Furthermore, we are able
to obtain explicit intervals of $\l$ such that (\ref{eq1}) has no,
one or two positive $\o$-periodic solution(s). A solution
$\vect{x}(t)=(x_1(t),...,x_n(t))$ is positive if, for each
$i=1,...,n,$ $x_i(t) \geq 0$ for all $t \in \mathbb{R}$ and there
is at least one component which is positive on $\mathbb{R}$.

We only give our main results here and their proofs are detailed in Appendix (\S \ref{proof periodicity}).

\begin{theorem}\label{th2} Assume \rm{(H1)-(H2)} hold. \\
(a). If $i_0=1$ or $2$,  then (\ref{eq1}) has  $i_0$ positive $\o$-periodic solution(s) for $\l>\f{1}{m(1)\G}>0$.\\
(b). If $i_{\infty}=1$ or $2$, then (\ref{eq1}) has  $i_{\infty}$ positive $\o$-periodic solution(s) for $0<\l <\f{1}{M(1)\chi}$.\\
(c). If $i_0=0$ , then (\ref{eq1}) has  no positive $\o$-periodic solution for sufficiently large $\lambda>0$; if $i_{\infty}=0$, then (\ref{eq1}) has  no positive $\o$-periodic solution for sufficiently small $\l >0$.\\
(d). If $i_0=i_{\infty}=1$ , then (\ref{eq1}) has a positive $\o$-periodic solution for all $\lambda>0.$
\end{theorem}

Theorem \ref{th3} is a direct consequence of the proof of Theorem \ref{th2}(c). Under the conditions of theorem \ref{th3} we are able to give explicit intervals of $\l$
such that (\ref{eq1}) has no positive $\o$-periodic solution.
\begin{theorem}\label{th3} Assume \rm{(H1)-(H2)} hold. \\
(a). If there is a $c_1>0$ such that $f^i(\vect{x}) \geq c_1 \norm{\vect{x}}$ for $ \vect{x} \in \mathbb{R}_+^n$,
$i=1,\dots,n$, then (\ref{eq1}) has no positive $\o$-periodic solution for $\lambda>\f{1}{\s \G c_1}.$\\
(b). If there is a $c_2>0$ such $f^i(\vect{x}) \leq c_2
\norm{\vect{x}} $ for $ \vect{x} \in \mathbb{R}_+^n$,
$i=1,\dots,n$,  then (\ref{eq1}) has no positive $\o$-periodic
solution for all $0< \l < \f{1}{c_2 \chi}$. \end{theorem}

\begin{theorem}\label{th4} Assume \rm{(H1)-(H2)} hold and $i_0=i_{\infty}=0$. If
\begin{equation}
\f{1}{ \s \G} \f{1}{\max\{\vect{F}_0, \vect{F}_{\infty}\}} < \l < \f{1}{\chi} \f{1}{\min \{\vect{F}_0, \vect{F}_{\infty} \}},
\end{equation}
then (\ref{eq1}) has a positive $\o$-periodic solution.
\end{theorem}

The above theorems for (\ref{eq1}) can be presented as Table \ref{tabledef} with different numbers of positive periodic solutions. We only deal with existence of positive solutions here and uniqueness will be addressed in a future work.
\setlength{\parskip}{0pt}

\setcounter{equation}{0}
\section{Stability of Positive Periodic Solutions}\label{stability sec}

In this section, we study the stability of the positive periodic solution for (\ref{eq1}) and start with the stability of zero solution since a non-zero solution can be transformed to a zero solution.  We establish sufficient and necessary criteria for the asymptotic stability  for the system. In this section, consider a delay system of the following form
\begin{equation}\label{SYSTEM}
\vect{x}'(t)=-\vect{A}(t)\vect{x}(t) + \vect{B}(t)\vect{F}(\vect{x}_t),
\end{equation}
where  $\vect{F}(\hat{0})=0$, $\vect{A}(t) = \mathop{\rm diag}[a_1(t),a_2(t),\dots,
a_n(t)]$, $\vect{B}(t) = \mathop{\rm diag}[b_1(t),b_2(t),\dots,
b_n(t)]$,
$a_i(t), b_i(t) \in C(\mathbb{R},\mathbb{R})$, $\vect{x}_t \in C([-r,0],\mathbb{R}^n)$. Here $r>0$ and $\hat{0}$ is the null element in $C([-r,0],\mathbb{R}^n)$.

We begin with some notation: $\mathbb{R}^n$ is an $n$-dimensional real Euclidean space with max norm $\abs{ \ \cdot \ }$; for $b>a \ ( {\rm it \ is \ may \ be \ possible \ } a=-\infty \ {\rm or} \ b=+\infty)$. We denote $C([a,b],\mathbb{R}^n)$ the Banach space of continuous functions mapping the interval $[a,b]$ into $\mathbb{R}^n$, and, for $\phi \in C([a,b],\mathbb{R}^n)$, the norm of $\phi$ is defined as $\norm{\phi}=\sup\limits_{a\leq\theta\leq b}\abs{\phi(\theta)}$, where $\abs{ \ \cdot \ }$ is a norm in $\mathbb{R}^n$. Here we use $\mathscr{C}=C([-r,0],\mathbb{R}^n)$, and $\vect{F}\in C(\mathscr{C},\mathbb{R}^n)$.

We assume that
\begin{description}
\item{(H3)~~$\il_0^ta_i(s)ds\rightarrow+\infty,  \ t\rightarrow +\infty  \ \ \ i=1, 2,\cdots, n.$}
\item{(H4)~~ There exist $L>0$ and $K_L>0$ such that
    $${\rm and }\ \  \abs{\vect{F}(\vect{\varphi})-\vect{F}(\vect{\psi})}\leq K_L\norm{\vect{\varphi}-\vect{\psi}},$$}
  for all $\vect{\varphi},\vect{\psi} \in \mathscr{C}_L:=\{\vect{x}\in \mathscr{C}:\norm{\vect{x}} \leq L \}$.
\item{(H5)~~There exists $\alpha\in(0,1)$ such that
$$ K_L \il_0^{t}\exp\ld-\il_s^{t}a_i(u)du\rd\abs{b_i(s)}ds\leq\alpha, \ \ \ t\geq0,\ i=1,\dots,n.$$}
\end{description}

The following two theorems give the delay-independent sufficient criteria and necessary criteria for the asymptotic stability, respectively. The details of the proof is presented in Appendix (\S \ref{proof stability}).

\begin{theorem}\label{STH} Assume \rm{(H3)-(H5)} hold, then the zero solution of (\ref{eq1}) is asymptotically stable.
\end{theorem}

\begin{theorem}\label{STH1} Assume \rm{(H4) and (H5)} hold, and also
  \begin{description}
  \item{(H6)$$\liminf_{t\rightarrow +\infty} \il_0^ta_i(s)ds\in\mathbb{R}, \ \ \ i=1, 2, \cdots, n.$$}
  \end{description}
If the zero solution of (\ref{SYSTEM}) is asymptotically stable, then we have
   \begin{description}
   \item{(H3)$$ \il_0^ta_i(s)ds\rightarrow+\infty, \ t\rightarrow +\infty \ \ \ i=1, 2, \cdots, n.$$}
   \end{description}
\end{theorem}


To study the asymptotic stability of the positive periodic solution $\vect{x}^\ast(t)$ of (\ref{SYSTEM}), let $\vect{y}(t)=\vect{x}(t)-\vect{x}^\ast(t),$ then
\begin{equation}\label{SYSTEM1}
\vect{y}'(t)=\vect{x}'(t)-(\vect{x}^{\ast}(t))'= -\vect{A}(t)\vect{y}(t) + \vect{B}(t)\vect{G}(\vect{y}_t),
\end{equation}
where
\be\label{capitalgfunctional}
\vect{G}(\phi):=\vect{F}(\phi+\vect{x}_t^*)-\vect{F}(\vect{x}_t^*).
\ee
The asymptotic stability of $\vect{x}^\ast(t)$ is equivalent to that of the zero solution of (\ref{SYSTEM1}).

\setcounter{equation}{0}
\setcounter{figure}{0}
\section{Application in Feedback Systems}\label{feedbacksyssec}
In this section, we consider the following system
\begin{equation}\label{Model1}
\left\{\begin{array}{l}
\displaystyle x' =  af(y)-bx\\
\displaystyle y' =  ce^{-(x+y)}-h(t)y,
\end{array}\right.
\end{equation}
where $h(t)$ is an $\o$-periodic function with $\bar{h} >0$,
$$f(y)=\f{\theta^2}{\theta^2+y^2}~~\hbox{or}~~f(y)=\f{y^2}{\theta^2+y^2},$$
and $a,b,c,\omega, \theta>0$. In the framework of \eqref{SYSTEM},
$$
\vect{A}(t)=\left(\ba{cc}b&0\\ 0&h(t)\ea\right),~~ \vect{B}(t)=\left(\ba{cc}a&0\\ 0&c\ea\right),~~\vect{F}(\phi)=\left(\ba{c}f(\phi_2(0))\\ \exp\{-(\phi_1(0)+\phi_2(0))\}\ea\right).$$

In 1977, Mackey and Glass \cite{MACKEYGLASS} considered a homogeneous population of mature circulating cells of density with two models of different types
$$\f{dP}{dt}=\f{\beta_0 \theta^n}{\theta^n+P_\tau^n}-\gamma P$$
and$$\f{dP}{dt}=\f{\beta_0 P_\tau^n}{\theta^n+P_\tau^n}-\gamma P.$$
System (\ref{Model1}) is said to be a negative growth feedback system if $f(y)=\theta^2/(\theta^2+y^2)$, and is called positive growth feedback system if  $f(p)=y^2/(\theta^2+y^2)$. \cite{Smith1995} includes extensive discussions on various results for the feedback systems.  System (\ref{Model1}) is an natural extension from the two models in \cite{MACKEYGLASS} when some parameters become periodic.  We apply our results for the existence and stability  to the positive periodic solutions of system (\ref{Model1}).  System (\ref{Model1}) with a delay can be discussed accordingly.

\subsection{Negative growth feedback system}
If $f(y)=\theta^2/(\theta^2+y^2)$, then one has
$$
\ba{c}
f_0^1=\infty,~~f_\infty^1=0,~~f_0^2=\infty,~~f_\infty^2=0,~~\vect{F}_0=\infty,~~\vect{F}_\infty=0,~~i_0=i_{\infty}=1.
\ea
$$
By Theorem \ref{th2}, we conclude that (\ref{Model1}) has one positive $\omega$-periodic solution $\vect{x}^\ast(t)=(x^\ast(t),y^\ast(t))$ (see Figure \ref{figexp2}) (a) and (b), respectively).

Next, we study the asymptotic stability of the positive $\omega$-periodic solution $\vect{x}^\ast(t)=(x^\ast(t), y^\ast(t))$ of (\ref{Model1}).  Let $\phi=(\phi_1, \phi_2)$, $\vect{G}(\phi)$ defined in \eqref{capitalgfunctional} now reads
$$\vect{G}(\phi)=\lx\f{\theta^2}{\theta^2+(\phi_2(0)+y^\ast(t))^2}-\f{\theta^2}{\theta^2+(y^\ast(t))^2}, \exp\{-[\phi_1(0)+\phi_2(0)]\}\rx^T.$$
Then,
\begin{equation*}
\ba{rcl}
\abs{\vect{G}(\phi)-\vect{G}(\psi)}
&=&\abs{\f{\theta^2}{\theta^2+(\phi_2(0)+y^\ast(t))^2}-\f{\theta^2}{\theta^2+(\psi_2(0)+y^\ast(t))^2}}\\\\
&&+\abs{\exp\{-[\phi_1(0)+\phi_2(0)]\}-\exp\{-[\psi_1(0)+\psi_2(0)]\}}\\\\
&\leq& \f{1}{\theta^2}\abs{\phi_2(0)-\psi_2(0)}\cdot \abs{\phi_2(0)+\psi_2(0)+2y^\ast(t)}\\\\
&&+\abs{\exp\{-[\phi_1(0)+\phi_2(0)]\}-\exp\{-[\psi_1(0)+\psi_2(0)]\}}\\\\
&\leq& \f{2L+2\hat{y}^*}{\theta^2}\abs{\phi_2(0)-\psi_2(0)}+
e^L(\abs{\phi_1(0)-\psi_1(0)}+\abs{\phi_2(0)-\psi_2(0)})\\
&\leq& \max\ld\f{2L+2\hat{y}^*}{\theta^2},e^L\rd\cdot(\abs{\phi_2(0)-\psi_2(0)}+\abs{\phi_1(0)-\psi_1(0)})\\
&\leq& \max\ld\f{2L+2\hat{y}^*}{\theta^2},e^L\rd\cdot\norm{\phi-\psi}\\
\ea
\end{equation*}
for $\phi, \psi\in \mathscr{C}_L$, where $\hat{y}^*=\max\limits_{t\in[0,\omega]}y^*(t)$.
That is, (H4) holds. Moreover,  (H5) reads
$$\max\ld\f{2L+2\hat{y}^*}{\theta^2},e^L\rd \il_0^t \exp\{-\il_s^t bdu\}ads\leq \alpha <1$$
and
$$\max\ld\f{2L+2\hat{y}^*}{\theta^2},e^L\rd \il_0^t \exp\{-\il_s^t h(u)du\}cds\leq \alpha <1,$$
i.e.,
\begin{equation}\label{feedbackcondd}
\ba{c}
a\max\ld\f{2L+2\hat{y}^*}{\theta^2},e^L\rd\il_0^t \exp\{b(s-t)\}ds\leq \alpha <1,\\
c\max\ld\f{2L+2\hat{y}^*}{\theta^2},e^L\rd\il_0^t \exp\ld-\il_s^t h(u)du\rd ds\leq \alpha <1.
\ea
\end{equation}
Then, by Theorem \ref{STH}, the positive $\omega$-periodic solution $\vect{x}^\ast(t)$ of (\ref{Model1}) is asymptotically stable if $a$ and $c$ are small enough so that \eqref{feedbackcondd} is satisfied (see Figure \ref{figexp3}).

\subsection{Positive growth feedback system}
If $f(y)=y^2/(\theta^2+y^2)$, then
$$
\ba{c}
f_0^1=0,~~f_\infty^1=0,~~f_0^2=\infty,~~f_\infty^2=0,~~\vect{F}_0=\infty,~~\vect{F}_\infty=0,~~i_0=i_{\infty}=1.
\ea
$$
Then, by Theorem \ref{th2},  \eqref{Model1} has one positive periodic solution.

For the stability of the positive solution of (\ref{Model1}), we first have
\begin{equation*}
\ba{rcl}
\abs{\vect{G}(\phi)-\vect{G}(\psi)}
&=&\abs{\f{(\phi_2(0)+y^\ast(t))^2}{\theta^2+(\phi_2(0)+y^\ast(t))^2}-\f{(\psi_2(0)+y^\ast(t))^2}{\theta^2+(\psi_2(0)+y^\ast(t))^2}}\\\\
&&+\abs{\exp\{-[\phi_1(0)+\phi_2(0)]\}-\exp\{-[\psi_1(0)+\psi_2(0)]\}}\\\\
&\leq& \f{1}{\theta^2}\abs{\phi_2(0)-\psi_2(0)}\cdot \abs{\phi_2(0)+\psi_2(0)+2y^\ast(t)}\\\\
&&+\abs{\exp\{-[\phi_1(0)+\phi_2(0)]\}-\exp\{-[\psi_1(0)+\psi_2(0)]\}}\\\\
&\leq& \max\ld\f{2L+2\hat{y}^*}{\theta^2},e^L\rd\cdot\norm{\phi-\psi}
\ea
\end{equation*}
for $\phi, \psi\in \mathscr{C}_L$. Then (H4) holds. Moreover, similar to the case of negative growth feedback, if \eqref{feedbackcondd} is satisfied, then $\vect{x}^\ast(t)$ of (\ref{Model1}) is asymptotically stable (see Fig. \ref{figexp12pic}).

To conclude, for either negative or positive growth feedback, when $a, c$ are small, (\ref{Model1}) has a positive $\omega$-periodic solution and is asymptotically stable. For large $a$ and $c$, numerical simulation shows that, although the stability criteria does not work, (\ref{Model1}) can still admit a positive periodic solution being asymptotically stable (see Figure \ref{figexp7}).

\setcounter{equation}{0}
\setcounter{figure}{0}
\section{Application in Delayed Periodic System}\label{fdeexpsec}
In this section, we present an example with delay to illustrate our main results.

Consider the following system of delay differential equations
\begin{equation}\label{eqexam}
\left\{ \begin{array}{ccc}
\dot{x}_1(t)=-a_1(t)x_1(t)+\lambda b_1(t)\exp\{-x_1(t-\tau)-x_2(t-\tau)\},\\
\dot{x}_2(t)=-a_2(t)x_2(t)+\lambda
b_2(t)\exp\{-x_1(t-\tau)-x_2(t-\tau)\},
\end{array} \right.
\end{equation}
where
$$\begin{array}{ccc}
a_1(t)=5+\sin(2\pi t),&~~b_1(t)=1+0.6\cos(2\pi t),\\
a_2(t)=5+\cos(2\pi t),&~~b_2(t)=1+0.5\sin(2\pi t).
\end{array}
$$

In 1974, Chow \cite{CHOW1} investigated the existence of positive periodic solution of the following difference-differential equation
$$N'(t)=-\mu N(t)+\rho \exp\{-\gamma N(t-r)\}$$
which is originally introduced by Wazewska-Czyzewska and Lasota
\cite{WAZLASTA} modeling the survival of red blood cells.

It is easily checked that
$$
\ba{c}
\bar{a}_i=5,~~\sigma_i=e^{-5},~~\sigma=e^{-5},~~\bar{b}_i=1,~~M(1)=1,~~m(1)=e^{-1},\\\Gamma=(e^5-1)^{-1},~~\chi=2e^5(e^5-1)^{-1},~~f_0^i=\infty,~~f_\infty^i=0
\ea
$$
Hence, $i_{\infty}=i_0=1$. By Theorem \ref{th2}, we can conclude that (\ref{eqexam}) has one positive 1-periodic solution $\vect{x}^\ast(t)$  for sufficiently small (see Figure \ref{fkwy21fig}(a) and \ref{fkwy22fig}) or large $\lambda$ (see Figure \ref{fkwy21fig}(b) and Figure \ref{fkwy23fig}), i.e.,
$$0<\lambda<\f{1}{M(1)\chi}=\f{e^5-1}{2e^5}\thickapprox 0.5~~\hbox{or}~~\lambda>\f{1}{m(1)\Gamma}=e(e^5-1)\thickapprox 400.7.$$

Next, we turn to study the stability of the positive solution of (\ref{eqexam}). Note that
\begin{equation*}
\begin{split}
\abs{\vect{G}(\vect{\phi})-\vect{G}(\vect{\psi})}
&=2e^{-(x_1^\ast(t-\tau)+x_2^\ast(t-\tau))} \abs{e^{-(\phi_1(-\tau)+\phi_2(-\tau))}
-e^{-(\psi_1(-\tau)+\psi_2(-\tau))}}\\
&\leq 2\abs{e^{-(\phi_1(-\tau)+\phi_2(-\tau))}
-e^{-(\psi_1(-\tau)+\psi_2(-\tau))}}\\
&\leq 2\norm{\phi-\psi}.
\end{split}
\end{equation*}
Then, (H4) holds. Moreover, for $t\geq 0$, if
\begin{equation}\label{conditioneq}
\ba{c}
2\lambda\il_0^t\exp\ld-\il_s^t(5+\sin(2\pi u))du\rd(1+0.6\cos(2\pi s))ds\leq\alpha<1\\
2\lambda\il_0^t \exp\ld-\il_s^t(5+\cos(2\pi u))du\rd(1+0.5\sin(2\pi s))ds\leq\alpha<1
\ea
\end{equation}
i.e.,
$$2\lambda \exp\ld-(5t+\f{\cos2\pi t}{2\pi})\rd\il_0^t \exp\ld 5s+\f{\cos2\pi s}{2\pi}\rd(1+0.6\cos2\pi s)ds\leq\alpha<1$$
and
$$2\lambda \exp\ld-(5t+\f{\sin2\pi t}{2\pi})\rd\il_0^t \exp\ld 5s+\f{\sin2\pi s}{2\pi}\rd(1+0.5\sin2\pi s)ds\leq\alpha<1$$
are satisfied, then $\vect{x}^\ast(t)$ is asymptotically stable (see Figure \ref{fkwy22fig}).

In summary, we conclude that, for $\lambda$ small enough so that \eqref{conditioneq} is satisfied, system (\ref{eqexam}) admits a positive periodic solution being asymptotically stable  (see Figure \ref{fkwy22fig}). For a large $\lambda$ not satisfying \eqref{conditioneq}, although the stability criteria does not work, the numerical simulations show that (\ref{eqexam}) has a positive periodic solution being asymptotically stable (see Figure \ref{fkwy23fig}).

By numerical simulations, we could conclude that, for small $\lambda$, the delay $\tau$ doesn't change the amplitude and period of the positive periodic solution of \eqref{eqexam}, but, for large $\lambda$, $\tau$ can change the period (see Figure \ref{delayeffectexmp}).

\section{Appendix}\label{appendixsec}
In the following, we devote to expounding the proof of the main findings on periodicity and asymptotic stability of \eqref{eq1}.

\subsection{Proof of periodicity}\label{proof periodicity}
In this appendix we give the proof of Theorem \ref{th2} and Theorem \ref{th4}. Our arguments as in \cite{CHENGZHANG,JIANGWEI2002,ReganWang2005,WANJIANG1,YE}
are based on a well-known fixed point theorem (Lemma~\ref{lm1}).
Such a method also has been employed in Wang \cite{HW,HWJMAA1} to
prove analogous results for the existence, multiplicity and
nonexistence of positive solutions of boundary value problems. We transform (\ref{eq1}) into a system of integral equations, and
then to the fixed point problem of an equivalent operator in a
cone. We establish several inequalities which allow us to
estimate the operator. Further, we apply the fixed
point index to show the existence, multiplicity and nonexistence
of positive $\o$-periodic solutions of (\ref{eq1}) based on the
inequalities.

\subsubsection{Preliminaries}
We recall some concepts and conclusions on the
fixed point index in a cone in \cite{DEIMING,GUOL,KRAS}. Let $X$
be a Banach space and $K$ be a closed, nonempty subset of $X$. $K$
is said to be a cone if $(i)$~$\alpha u+\beta v\in K $ for all
$u,v\in K$ and all $\alpha,\beta>0$ and $(ii)$~$u,-u\in K$ imply
$u=0$. Assume $\O$ is a bounded open subset in $K$ with the
boundary $\partial \O$, and let $T: K \cap \overline{\O} \to K$ is
completely continuous such that $Tx  \neq x $ for $x \in \partial
\O \cap K$, then the fixed point index $i(T, K \cap \O, K)$ is
defined. If $i(T, K \cap \O, K) \neq 0$, then $T$ has a fixed
point in $K \cap \O$. The following well-known result of the fixed
point index is crucial in our arguments.

\begin{lemma}\label{lm1} {\rm (\cite{DEIMING,GUOL,KRAS}).} Let $X$ be a
Banach space and $K$ a cone in $X$. For $r >0$, define $K_r =\{x
\in K: \|x\| < r \}$. Assume that $T: \overline{K}_r \to K$ is
completely continuous such that $Tx  \neq x $ for $x \in \partial
K_r= \{x \in K: \|x\| = r \}$.

     {\rm(i)} If $\|Tx\| \geq \|x\|$ for $x \in \partial K_r$, then $i(T, K_r, K)=0.$

     {\rm(ii)} If $\|Tx\| \leq \|x\|$ for $x \in \partial K_r$, then $i(T, K_r, K)=1.$
\end{lemma}

In order to apply Lemma \ref{lm1} to (\ref{eq1}), let $X$ be the
Banach space defined by
$$X=\{\vect{x}(t)\in C(\mathbb{R},\mathbb{R}^n):
\vect{x}(t+\o)=\vect{x}(t), t \in \mathbb{R}, i=1,\dots,n\}$$ with
a norm $\displaystyle{\norm{\vect{x}}= \sum_{i=1}^n
\sup_{t\in[0,\o]} \abs{x_i(t)}},$ for $\vect{x}=(x_1,...,x_n) \in
X.$  For $\vect{x} \in X$ or $\mathbb{R}^n_+$, $\norm{\vect{x}}$
denotes the norm of $\vect{x}$ in $X$ or $\mathbb{R}^n_+$,
respectively.

Define
$$
K = \{\vect{x}=(x_1,\dots,x_n) \in X: x_i(t) \geq \s_i \sup_{t\in[0,\o]} \abs{x_i(t)}, i=1,\dots,n, t \in [0, \o] \}.
$$
It is clear $K$ is cone in $X$. For $r>0$, define $\O_r =
\{\vect{x} \in K: \norm{\vect{x}} < r \}. $ It is clear that
$\partial \O_r = \{\vect{x} \in K: \norm{\vect{x}}=r\}$. Let
$\vect{T}_{\l}: K \to X$ be a map with components
$(T_{\l}^1,\dots,T_{\l}^n)$:
\begin{equation}\label{T_def}
T_{\l}^i\vect{x}(t) = \l \il^{t+\o}_t G_i(t, s) b_i(s)
f^i(\vect{x}(s-\t(s)))ds,~~i=1,\dots,n,
\end{equation}
where
$$
G_i(t,s)=\f{1}{\s^{-1}_i-1}\exp\ld\il_t^s a_i(\theta)d\theta\rd
$$
satisfying
$$
\f{1}{\s^{-1}_i-1} \leq G_i(t,s) \leq \f{\s^{-1}_i}{\s^{-1}_i-1},~~t \leq s \leq t+\o.
$$
\begin{lemma}\label{lm-compact}
Assume \rm{(H1)-(H2)} hold. Then $\vect{T} _{\l}(K) \subset K$ and $\vect{T}_{\l}: K \to K$ is compact and continuous.
\end{lemma}
\begin{proof}
In view of the definition of $K$, for $\vect{x} \in K$, we have, $i=1,\dots,n,$
\begin{equation*}
\begin{split}
(T_{\l}^i\vect{x})(t+\o)
& =\l \il^{t+2\o}_{t+\o} G_i(t+\o, s) b_i(s) f^i(\vect{x}(s-\t(s)))ds \\
& =\l \il^{t+\o}_{t} G_i(t+\o, \theta+\o) b_i(\theta+\o) f^i(\vect{x}(\theta+\o-\t(\theta+\o)))d\theta \\
& =\l \il^{t+\o}_{t} G_i(t, s) b_i(s) f^i(\vect{x}(s-\t(s)))ds \\
& = (T_{\l}^i\vect{x})(t).
\end{split}
\end{equation*}
It is easy to see that $\il^{t+\o}_{t} b_i(s) f^i(\vect{x}(s-\t(s)))ds$ is a constant because of the periodicity of $b_i(t) f^i(\vect{x}(t-\t(t)))$.
One can show that, for $\vect{x} \in K$ and $t \in [0,\o]$, $i=1,\dots,n,$
\begin{equation*}
\begin{split}
T_{\l}^i\vect{x}(t)
& \geq \f{1}{\s^{-1}_i-1}\l \il^{t+\o}_{t} b_i(s) f^i(\vect{x}(s-\t(s)))ds \\
& = \f{1}{\s^{-1}_i-1}\l \il^{\o}_{0} b_i(s) f^i(\vect{x}(s-\t(s)))ds \\
& = \s_i \f{\s^{-1}_i}{\s^{-1}_i-1} \l \il^{\o}_{0} b_i(s) f^i(\vect{x}(s-\t(s)))ds \\
&  \geq \s_i  \sup_{t\in[0,\o]} \abs{T_{\l}^i\vect{x}(t)}.
\end{split}
\end{equation*}
Thus $\vect{T} _{\l}(K) \subset K$ and it is easy to show that
$\vect{T}_{\l}: K \to K$ is compact and continuous. \end{proof}
\begin{lemma}\label{lm-fixed-equation-equal} Assume that \rm{(H1)-(H2)}
hold. Then $\vect{x}\in K$ is a positive periodic solution of
(\ref{eq1}) if and only if it is a fixed point of $\vect{T} _{\l}$
in $K.$

\end{lemma} \begin{proof} If $\vect{x}=(x_1,\dots,x_n) \in K$ and
$\vect{T}_{\l}\vect{x}=\vect{x}$, then, for $i=1,\dots,n,$
\begin{equation*}
\begin{split}
x_i'(t)
& = \f{d}{dt} (\l \il^{t+\o}_t G_i(t, s) b_i(s) f^i(\vect{x}(s-\t(s)))ds)\\
& =  \l G_i(t, t+\o) b_i(t+\o)f^i(\vect{x}(t+\o-\t(t+\o))\\
& \quad - \l G_i(t,t)b_i(t)f^i(\vect{x}(t-\t(t)))
- a_i(t) T_{\l}^ix(t) \\
& =  \l [G_i(t, t+\o)-G_i(t,t)]b_i(t)f^i(\vect{x}(t-\t(t)))-a_i(t) T_{\l}^ix(t)\\
& = - a_i(t) x_i(t)+ \l b_i(t)f^i(\vect{x}(t-\t(t))).
\end{split}
\end{equation*}
Thus $\vect{x}$ is a positive $\o$-periodic solution of (\ref{eq1}). On the other hand, if $\vect{x}=(x_1,\dots,x_n)$ is a positive $\o$-periodic function, then
$ \l b_i(t)f^i(\vect{x}(t-\t(t)))= a_i(t) x_i(t)+x_i'(t)$ and
\begin{equation*}
\begin{split}
T_{\l}^i\vect{x}(t)
& = \l \il^{t+\o}_t G_i(t, s) b_i(s) f^i(\vect{x}(s-\t(s)))ds\\
& =  \il^{t+\o}_t G_i(t, s) (a_i(s) x_i(s)+x_i'(s))ds \\
& =  \il^{t+\o}_t G_i(t, s) a_i(s) x_i(s)ds +\il^{t+\o}_t G_i(t, s) x'_i(s)ds \\
& =  \il^{t+\o}_t G_i(t, s) a_i(s) x_i(s)ds + G_i(t, s) x_i(s)|^{t+\o}_t - \il^{t+\o}_t G_i(t, s) a_i(s) x_i(s)ds \\
& = x_i(t).
\end{split}
\end{equation*}
Thus, $\vect{T}_{\l}\vect{x}=\vect{x}$, Furthermore, in view of
the proof of Lemma \ref{lm-compact}, we also have $x_i(t) \geq
\s_i \sup_{t\in[0,\o]} x_i(t)$ for $t \in [0,\o].$ That is,
$\vect{x}$ is a fixed point of $\vect{T}_{\l}$ in $K$. \end{proof}

\begin{lemma}\label{f_estimate_>*}
Assume that \rm{(H1)-(H2)} hold. For
any $\eta > 0$ and $ \vect{x}=(x_1, \dots, x_n) \in K $, if there
exists a component $f^i$ of $\vect{F}$ such that
\mbox{$f^i(\vect{x}(t)) \geq \sum_{j=1}^n x_j(t)\eta$ } for $ t
\in [0, \o]$, then $ \norm{\vect{T}_{\l}\vect{x}} \geq \l \s \G
\eta \norm{\vect{x}}. $
\end{lemma}

\begin{proof} Since $\vect{x} \in K$ and
\mbox{$f^i(\vect{x}(t)) \geq \sum_{j=1}^n x_j(t)\eta$ } for $ t
\in [0, \o]$, we have
\begin{equation*}
\begin{split}
(T_{\l}^i\vect{x})(t)
& \geq \f{1}{\s^{-1}_i-1}\l \il^{\o}_{0} b_i(s) f^i(\vect{x}(s-\t(s)))ds \\
& \geq \f{1}{\s^{-1}_i-1}\l \il^{\o}_{0} b_i(s) \sum_{j=1}^n x_j(s-\t(s)) \eta ds \\
& \geq \f{1}{\s^{-1}_i-1}\l \il^{\o}_{0} b_i(s) ds \sum_{j=1}^n \s_j\sup_{t\in[0,\o]} x_j(t) \eta  \\
& \geq \l \eta  \norm{\vect{x}}\min\limits_{j=1,\dots,n}\{\s_j\} (\s^{-1}_i-1)^{-1}{\il^{\o}_{0} b_i(s) ds}.
\end{split}
\end{equation*}
Thus $\norm{\vect{T}_{\l}\vect{x}} \geq \l \s \G \eta
\norm{\vect{x}}$. \end{proof} \begin{lemma}\label{f_estimate_<*} Assume that
\rm{(H1)-(H2)} hold. For any $ r >0 $ and $\vect{x}=(x_1, \dots,
x_n) \in \partial\O_{r}$, if there exists an $\varepsilon > 0$
such that $f^i(\vect{x}(t)) \leq \varepsilon \sum_{j=1}^n x_j(t)$,
$i=1,\dots,n,$ for $ t \in [0, \o]$, then $
\norm{\vect{T}_{\l}\vect{x}} \leq \l \chi \varepsilon
\norm{\vect{x}} $ \end{lemma} \begin{proof} From the definition of $\vect{T}$, for
$\vect{x} \in \partial\O_{r}$, we have
\begin{equation*}
\begin{split}
\norm{\vect{T}_{\l}\vect{x}}
  & \leq  \sum_{i=1}^n \f{\s^{-1}_i}{\s^{-1}_i-1} \l \il^{\o}_{0} b_i(s) f^i(\vect{x}(s-\t(s)))ds  \\
  & \leq   \sum_{i=1}^n \f{\s^{-1}_i}{\s^{-1}_i-1} \l \il^{\o}_{0} b_i(s)ds  \varepsilon  \norm{\vect{x}}=  \l \chi \varepsilon \norm{\vect{x}}.
\end{split}
\end{equation*}\end{proof}
In view of the definitions of $m(r)$ and $M(r)$,  it follows that
$ M(r) \geq f^i(\vect{x}(t)) \geq  m(r)$ \; $\rm{for}\; t \in [0,
\o]$, $i=1,\dots,n$ if $\vect{x} \in \partial \O_{r}$, $r>0$ .
Thus it is easy to see that the following two lemmas can be shown
in similar manners as in Lemmas \ref{f_estimate_>*} and
\ref{f_estimate_<*}.

\begin{lemma}\label{lm8} Assume \rm{(H1)-(H2)} hold. If $ \vect{x} \in
\partial \O_{r}$, $r
>0$, then $ \norm{\vect{T}_{\l}\vect{x}}   \geq \l \G m(r). $\end{lemma}


\begin{lemma}\label{lm9} Assume \rm{(H1)-(H2)} hold. If $ \vect{x}\in
\partial \O_{r}$, $r >0$, then $ \norm{\vect{T}_{\l}\vect{x}}
\leq \l \chi M(r). $ \end{lemma}


\subsubsection{Proof of Theorem \ref{th2}}

Part (a). Choose a number $r_1 =1$. By Lemma \ref{lm8}, we
have
$$
\norm{\vect{T}_{\l}\vect{x}}  > \norm{\vect{x}}  , \;\;
\rm{for}\;\;\vect{x}\in  \partial \O_{r_1} \ and \  \l >
\f{1}{m(r_1)\Gamma}=\f{1}{m(1)\Gamma}.
$$
If $\vect{F}_0=0$, then $f^i_0=0$, $i=1,\dots,n$.  And we can choose $0 <r_2 < r_1$ so that  $f^i(\vect{x}) \leq \varepsilon  \norm{\vect{x}} $ for $\vect{x}\in \mathbb{R}_+^n$ and $\norm{\vect{x}} \leq r_2$,
$i=1,\dots,n,$ where the constant $\varepsilon> 0$ satisfies
$
\l \varepsilon \chi < 1.
$
Thus $f^i(\vect{x}(t)) \leq  \varepsilon \sum_{j=1}^n x_j(t)$, $i=1, \dots, n,$ for $\vect{x}=(x_1, \dots, x_n) \in \partial\O_{r_2}$ and $t \in [0, \o]$.
We have by Lemma ~\ref{f_estimate_<*} that
$$
 \norm{\vect{T}_{\l}\vect{x}}  \leq \l \varepsilon \chi \norm{\vect{x}}   < \norm{\vect{x}}   \quad \textrm{for} \quad  \vect{x}\in \partial\O_{r_2}.
$$
It follows from Lemma ~\ref{lm1} that
$$
i(\vect{T}_{\l}, \O_{r_1}, K)=0, \quad i(\vect{T}_{\l}, \O_{r_2}, K)=1.
$$
Thus $i(\vect{T}_{\l}, \O_{r_1} \setminus \bar{\O}_{r_2}, K)=-1$
and $\vect{T}_{\l}$ has a fixed point in  $\O_{r_1} \setminus
\bar{\O}_{r_2}$, which is a positive $\o$-periodic solution of
(\ref{eq1}) for $ \l > 1/(m(1)\Gamma)$.

If $\vect{F}_{\infty}=0$, then $f^i_{\infty}=0$, $i=1,\dots,n$.
And there is an $\hat{H}>0$ such that $f^i(\vect{x}) \leq
\varepsilon \norm{\vect{x}}$ for $\vect{x}\in \mathbb{R}_+^n$ and
$\norm{\vect{x}} \geq \hat{H}$, where the constant $\varepsilon >
0$ satisfies $ \l \varepsilon \chi < 1. $ Let $r_3=\max\{2r_1,
\f{\hat{H}}{\s}\}$ and it follows that $\sum_{i=1}^n x_i(t)
\geq \s \norm{\vect{x}}  \geq \hat{H}$ for $\vect{x}=(x_1, \dots,
x_n) \in \partial\O_{r_3}$ and $t \in [0, \o]$. Thus
$f^i(\vect{x}(t)) \leq  \varepsilon \sum_{i=1}^n x_i(t) $ for
$\vect{x}=(x_1,\dots,x_n) \in \partial\O_{r_3}$ and $t \in [0,
\o]$. In view of Lemma ~\ref{f_estimate_<*}, we have
$$ \norm{\vect{T}_{\l}\vect{x}}   \leq \l \varepsilon \chi \norm{\vect{x}}   <  \norm{\vect{x}}   \quad \textrm{for} \quad  \vect{x}\in \partial\O_{r_3}.$$
Again, it follows from Lemma ~\ref{lm1} that
$$
i(\vect{T}_{\l}, \O_{r_1}, K)=0, \quad \quad i(\vect{T}_{\l}, \O_{r_3}, K)=1.
$$
Thus $i(\vect{T}_{\l}, \O_{r_3} \setminus \bar{\O}_{r_1}, K)=1$,
and (\ref{eq1}) has a positive $\o$-periodic solution for $ \l >
1/(m(1)\Gamma)$.

If $\vect{F}_0=\vect{F}_{\infty}=0$, it is easy to see from the above proof that  $\vect{T}_{\l}$ has a fixed point $\vect{x}_1$ in  $\O_{r_1} \setminus \bar{\O}_{r_2}$
and a fixed point $\vect{x}_2$ in $\O_{r_3} \setminus \bar{\O}_{r_1}$ such that
$$
r_2 < \norm{\vect{x}_1} < r_1 < \norm{\vect{x}_2} < r_3.
$$
Consequently, (\ref{eq1}) has two positive $\o$-periodic solutions for $ \l > 1/(m(1)\Gamma)$ if $\vect{F}_0=\vect{F}_{\infty}=0$.

Part (b). Choose a number $r_1 =1$. By Lemma \ref{lm9}, one has
$$
\norm{\vect{T}_{\l}\vect{x}}  < \norm{\vect{x}}  , \;\;
\rm{for}\;\;\vect{x}\in  \partial \O_{r_1} \ \rm{and} \   0 < \l <
\f{1}{\chi M(r_1)}=\f{1}{M(1) \chi} .
$$
If $\vect{F}_0=\infty$, then there exists a component of $\vect{F}$
such that $f^i_0 = \infty$. Thus there is a positive number $r_2 <
r_1$ such that $f^i(\vect{x}) \geq \eta \norm{\vect{x}}$ for
$\vect{x}\in \mathbb{R}_+^n$ and $\norm{\vect{x}} \leq r_2$,
where $\eta > 0$ is chosen so that $\l  \G \eta \s > 1.$ Then
$$
f^i(\vect{x}(t)) \geq \eta \sum_{j=1}^n x_j(t) \;\; {\rm for} \;\; \vect{x}=(x_1, \dots, x_n) \in \partial \O_{r_2}, \;\;t \in [0,\o].
$$
Lemma ~\ref{f_estimate_>*}
implies that
$$
\norm{\vect{T}_{\l}\vect{x}}  \geq \l \G \eta \norm{\vect{x}} \s  >\norm{\vect{x}}   \quad \textrm{for} \quad  \vect{x}\in \partial\O_{r_2}.
$$
It follows from Lemma ~\ref{lm1} that
$$
i(\vect{T}_{\l}, \O_{r_1}, K)=1, \quad i(\vect{T}_{\l}, \O_{r_2}, K)=0.
$$
Thus $i(\vect{T}_{\l}, \O_{r_1} \setminus \bar{\O}_{r_2}, K)=1$
and $\vect{T}_{\l}$ has a fixed point in  $\O_{r_1} \setminus
\bar{\O}_{r_2}$  for $0< \l < 1/(M(1) \chi) $, which is a
positive $\o$-periodic solution of (\ref{eq1}).

If $\vect{F}_{\infty}=\infty$, there is a component of $\vect{F}$
such that $f^i_{\infty} = \infty$. Therefore there is an $\hat{H}
> 0$ such that $f^i(\vect{x}) \geq  \eta \norm{\vect{x}}$ for
$\vect{x}\in \mathbb{R}_+^n$ and $\vect{x} \geq \hat{H}$ , where
$\eta > 0$ is chosen so that $\l  \G \eta \s > 1.$ Let $r_3 =
\max\{2r_1, \hat{H}/\s\}$. If $ \vect{x}=(x_1, \dots,
x_n) \in \partial \O_{r_3}$, then
$$\min\limits_{0 \leq t \leq  \o} \sum_{j=1}^n x_j(t) \geq \s
\norm{\vect{x}}   \geq \hat{H},$$
and hence,
$$
f^i(\vect{x}(t)) \geq   \eta \sum_{j=1}^nx_j(t) \;\; {\rm for}\;\; t \in [0, \o].
$$
Again, it follows from Lemma  \ref{f_estimate_>*} that
$$
\norm{\vect{T}_{\l}\vect{x}}  \geq \l  \G \eta \s \norm{\vect{x}}   > \norm{\vect{x}}   \quad \textrm{for} \quad  \vect{x}\in \partial\O_{r_3}.
$$
It follows from Lemma ~\ref{lm1} that
$$
i(\vect{T}_{\l}, \O_{r_1}, K)=1, \quad i(\vect{T}_{\l}, \O_{r_3}, K)=0,
$$
and hence, $i(\vect{T}_{\l}, \O_{r_3} \setminus \bar{\O}_{r_1},
K)=-1$. Thus, $\vect{T}_{\l}$ has a fixed point in $\O_{r_3}
\setminus \bar{\O}_{r_1}$  for $0< \l < 1/{(M(1) \chi)} $,
which is a positive $\o$-periodic solution of (\ref{eq1}).

If $\vect{F}_0=\vect{F}_{\infty}=\infty$, it is easy to see from the above proof that  $\vect{T}_{\l}$ has a fixed point $\vect{x}_1$ in  $\O_{r_1} \setminus \bar{\O}_{r_2}$
and a fixed point $\vect{x}_2$ in $\O_{r_3} \setminus \bar{\O}_{r_1}$ such that
$$
r_2 < \norm{\vect{x}_1} < r_1 < \norm{\vect{x}_2} < r_3.
$$
Consequently, (\ref{eq1}) has two positive $\o$-periodic solutions
for $0< \l < 1/{(M(1) \chi)} $ if
$\vect{F}_0=\vect{F}_{\infty}=\infty$.

Part (c). If $i_0=0$, then $\vect{F}_0>0$ and $\vect{F}_{\infty}>0$. Therefore there exist two components $f^i$ and $f^j$ of $\vect{F}$ and positive numbers
$\eta_1$, $ \eta_2$, $r_1$ and $r_2$ such that $r_1 < r_2$ and
\begin{eqnarray*}
  f^i(\vect{x})& \geq & \eta_1 \norm{\vect{x}} \quad \textrm{for} \quad \vect{x}\in \mathbb{R}_+^n,\; \norm{\vect{x}} \leq r_1,  \\[.2cm]
  f^j(\vect{x})& \geq &  \eta_2 \norm{\vect{x}} \quad \textrm{for} \quad \vect{x}\in \mathbb{R}_+^n,\; \norm{\vect{x}} \geq r_2.
\end{eqnarray*}
Let
$$c_1 = \min\ld\eta_1, \eta_2, \min\ld\f{f^j(\vect{x})}{\var(\norm{\vect{x}})}: \vect{x}\in \mathbb{R}_+^n, \;\; r_1 \s \leq \norm{\vect{x}} \leq  r_2\rd\rd > 0.$$
Then
\begin{align} \label{ineq1}
  f^i(\vect{x}) \geq c_1 \norm{\vect{x}}\; \textrm{ for } \;\vect{x}\in \mathbb{R}_+^n,\; \norm{\vect{x}} \leq r_1,
\end{align}
and
\begin{align} \label{ineq2}
  f^j(\vect{x})   \geq   c_1 \norm{\vect{x}} \; \textrm{ for } \;\vect{x}\in \mathbb{R}_+^n,\; \norm{\vect{x}} \geq r_1 \s.
\end{align}

Assume $\vect{v}(t)=(v_1,\dots,v_n)$ is a positive $\o$-periodic solution of (\ref{eq1}). We will show that this leads to a contradiction
for $\l > \l_0$, where $\l_0={1}/{(\s \G c_1)}  $. In fact, if $\norm{\vect{v}} \leq r_1$, (\ref{ineq1}) implies that
$$
f^i(\vect{v}(t)) \geq c_1\sum_{j=1}^n v_j(t), \;\; {\rm for }\;\; t \in [0,\o].
$$
On the other hand, if $\norm{\vect{v}} > r_1$, then $ \min_{ 0 \leq t \leq \o} \sum_{i=1}^n v_i(t) \geq
\s \norm{\vect{v}} > r_1 \s,$
which, together with (\ref{ineq2}), implies that
$$
f^j(\vect{v}(t)) \geq c_1 \sum_{m=1}^n v_m(t), \;\; {\rm for }\;\; t \in [0, \o].
$$
Since  $\vect{T}_{\l}\vect{v}(t) =\vect{v}(t)$ for $ t \in [0,\o]$, it
follows from Lemma ~\ref{f_estimate_>*} that, for $\l > \l_0,$
$$
  \norm{\vect{v}}  =  \norm{\vect{T}_{\l}\vect{v} }\geq \l  \s \G c_1 \norm{\vect{v}} > \norm{\vect{v}},
$$
which is a contradiction.

If $i_{\infty}=0$, then $\vect{F}_0  < \infty$ and $\vect{F}_{\infty} < \infty$. Thus $f^i_{0}<\infty$ and $f^i_{\infty}<\infty$, $i=1,...,n.$
Therefore, for each $i=1,...,n,$ there exist positive numbers
$\e_1^i$, $ \e_2^i$, $r_1^i$ and $r_2^i$ such that $r_1^i < r_2^i$,
$$
  f^i(\vect{x})  \leq   \e_1^i \norm{\vect{x}}\; \rm{for} \;\vect{x}\in \mathbb{R}_+^n,\; \norm{\vect{x}} \leq r_1^i,
$$
and
$$
  f^i(\vect{x})   \leq  \e_2^i \norm{\vect{x}}\; \rm{for} \;\vect{x}\in \mathbb{R}_+^n,\; \norm{\vect{x}} \geq r_2^i.
$$
Let
$$
\e^i = \max\{\e_1^i, \e_2^i, \max\{\f{f^i(\vect{x})}{\norm{\vect{x}}}: \vect{x}\in \mathbb{R}_+^n, \;\; r_1^i \leq \norm{\vect{x}} \leq  r_2^i \}\} > 0
$$
and $ c_2 = \max\limits_{i=1,...,n} \{\e^i\} > 0.$ Thus, we have
$$
 f^i(\vect{x}) \leq  c_2 \norm{\vect{x}}\; \rm{ for } \;\vect{x}\in \mathbb{R}_+^n,\;\;i=1,...,n.
$$
Assume $\vect{v}(t)$ is a positive $\o$-periodic solution of (\ref{eq1}). We will show that this leads to a contradiction
for $0<  \l < \l_0$, where $\l_0 = {1}/{(c_2 \chi)}.$
Since  $\vect{T}_{\l}\vect{v}(t) = \vect{v}(t)$ for $ t \in [0,\o]$, it follows from Lemma \ref{f_estimate_<*} that , for $0< \l <\l_0,$
\begin{equation*}
\begin{split}
  \norm{\vect{v}}
   =   \vect{\vect{T}_{\l}\vect{v}}\leq \l \chi c_2 \norm{\vect{v}} <  \norm{\vect{v}},
\end{split}
\end{equation*}
which is a contradiction. The proof is complete.

Part (d).  If $i_0=i_{\infty}=1$, then $F_0=0, F_{\infty}=\infty$ or $F_0=\infty, F_{\infty}=0$.  For $F_0=0, F_{\infty}=\infty$, then $f^i_0=0$, $i=1,\dots,n$.  And we can choose $0 < r_1$ so that  $f^i(\vect{x}) \leq \varepsilon  \norm{\vect{x}} $ for $\vect{x}\in \mathbb{R}_+^n$ and $\norm{\vect{x}} \leq r_1$,
$i=1,\dots,n,$ where the constant $\varepsilon> 0$ satisfies
$
\l \varepsilon \chi < 1.
$
Thus $f^i(\vect{x}(t)) \leq  \varepsilon \sum_{j=1}^n x_j(t)$, $i=1, \dots, n,$ for $\vect{x}=(x_1, \dots, x_n) \in \partial\O_{r_1}$ and $t \in [0, \o]$.
We have by Lemma ~\ref{f_estimate_<*} that
$$
 \norm{\vect{T}_{\l}\vect{x}}  \leq \l \varepsilon \chi \norm{\vect{x}}   < \norm{\vect{x}}   \quad \textrm{for} \quad  \vect{x}\in \partial\O_{r_1}.
$$
On the other hand, If $\vect{F}_{\infty}=\infty$, there is a component of $\vect{F}$
such that $f^i_{\infty} = \infty$. Therefore there is an $\hat{H}
> 0$ such that $f^i(\vect{x}) \geq  \eta \norm{\vect{x}}$ for
$\vect{x}\in \mathbb{R}_+^n$ and $\vect{x} \geq \hat{H}$ , where
$\eta > 0$ is chosen so that $\l  \G \eta \s > 1.$ Let $r_2 =
\max\{2r_1, \hat{H}/\s\}$. If $ \vect{x}=(x_1, \dots,
x_n) \in \partial \O_{r_2}$, then
$$\min\limits_{0 \leq t \leq  \o} \sum_{j=1}^n x_j(t) \geq \s
\norm{\vect{x}}   \geq \hat{H},$$
and hence,
$$
f^i(\vect{x}(t)) \geq   \eta \sum_{j=1}^nx_j(t) \;\; {\rm for}\;\; t \in [0, \o].
$$
Again, it follows from Lemma  \ref{f_estimate_>*} that
$$
\norm{\vect{T}_{\l}\vect{x}}  \geq \l  \G \eta \s \norm{\vect{x}}   > \norm{\vect{x}}   \quad \textrm{for} \quad  \vect{x}\in \partial\O_{r_2}.
$$
It follows from Lemma ~\ref{lm1} that
$$
i(\vect{T}_{\l}, \O_{r_1}, K)=1, \quad i(\vect{T}_{\l}, \O_{r_2}, K)=0,
$$
and hence, $i(\vect{T}_{\l}, \O_{r_2} \setminus \bar{\O}_{r_1},
K)=-1$. Thus, $\vect{T}_{\l}$ has a fixed point in $\O_{r_2}
\setminus \bar{\O}_{r_1}$, which is a positive $\o$-periodic solution of (\ref{eq1}).

Now let's look at the case $F_0=\infty, F_{\infty}=0$. There exists a component of $\vect{F}$
such that $f^i_0 = \infty$. Thus there is a positive number $r_1$ such that $f^i(\vect{x}) \geq \eta \norm{\vect{x}}$ for
$\vect{x}\in \mathbb{R}_+^n$ and $\norm{\vect{x}} \leq r_1$,
where $\eta > 0$ is chosen so that $\l  \G \eta \s > 1.$ Then
$$
f^i(\vect{x}(t)) \geq \eta \sum_{j=1}^n x_j(t) \;\; {\rm for} \;\; \vect{x}=(x_1, \dots, x_n) \in \partial \O_{r_1}, \;\;t \in [0,\o].
$$
Lemma ~\ref{f_estimate_>*}
implies that
$$
\norm{\vect{T}_{\l}\vect{x}}  \geq \l \G \eta \norm{\vect{x}} \s  >\norm{\vect{x}}   \quad \textrm{for} \quad  \vect{x}\in \partial\O_{r_1}.
$$

On the other hand,  because $\vect{F}_{\infty}=0$, then $f^i_{\infty}=0$, $i=1,\dots,n$.
And there is an $\hat{H}>0$ such that $f^i(\vect{x}) \leq
\varepsilon \norm{\vect{x}}$ for $\vect{x}\in \mathbb{R}_+^n$ and
$\norm{\vect{x}} \geq \hat{H}$, where the constant $\varepsilon >
0$ satisfies $ \l \varepsilon \chi < 1. $ Let $r_2=\max\{2r_1,
\f{\hat{H}}{\s}\}$ and it follows that $\sum_{i=1}^n x_i(t)
\geq \s \norm{\vect{x}}  \geq \hat{H}$ for $\vect{x}=(x_1, \dots,
x_n) \in \partial\O_{r_2}$ and $t \in [0, \o]$. Thus
$f^i(\vect{x}(t)) \leq  \varepsilon \sum_{i=1}^n x_i(t) $ for
$\vect{x}=(x_1,\dots,x_n) \in \partial\O_{r_2}$ and $t \in [0,
\o]$. In view of Lemma ~\ref{f_estimate_<*}, we have
$$ \norm{\vect{T}_{\l}\vect{x}}   \leq \l \varepsilon \chi \norm{\vect{x}}   <  \norm{\vect{x}}   \quad \textrm{for} \quad  \vect{x}\in \partial\O_{r_2}.$$
Again, it follows from Lemma ~\ref{lm1} that
$$
i(\vect{T}_{\l}, \O_{r_1}, K)=0, \quad \quad i(\vect{T}_{\l}, \O_{r_2}, K)=1.
$$
Thus $i(\vect{T}_{\l}, \O_{r_2} \setminus \bar{\O}_{r_1}, K)=1$,
and (\ref{eq1}) has a positive $\o$-periodic solution.

\subsubsection{Proof of Theorem \ref{th4}}
If $\vect{F}_{\infty}>\vect{F}_0$, then ${1}/{(\s \G \vect{F}_{\infty})} < \l < {1}/{(\chi\vect{F}_0)}$.
It is easy to see that there
exists an $0<\varepsilon <\vect{F}_{\infty}$ such that
$$
\f{1}{\s \G (\vect{F}_{\infty}- \varepsilon)} < \l <\f{1}{\chi (\vect{F}_0 +\varepsilon)}.
$$
Now, according to the definition of $F_0$, there is an $r_1>0$ such that $f^i(\vect{x}) \leq (\vect{F}_0 + \varepsilon) \norm{\vect{x}}$ for
$ \vect{x} \in \mathbb{R}_+^n$ and $\norm{\vect{x}} \leq r_1$, $i=1, \dots, n$.
Thus $f^i(\vect{x}(t)) \leq  (\vect{F}_0 + \varepsilon) \sum_{j=1}^n x_j(t)$, $i=1, \dots, n$,  for $\vect{x}=(x_1,\dots, x_n) \in \partial\O_{r_1}$ and $t \in [0, \o]$.
We have by Lemma ~\ref{f_estimate_<*} that
$$
 \norm{\vect{T}_{\l}\vect{x}}  \leq \l(\vect{F}_0 + \varepsilon) \chi \norm{\vect{x}}   < \norm{\vect{x}}   \quad \textrm{for} \quad  \vect{x} \in \partial\O_{r_1}.
$$
On the other hand, there is a fixed i such that $f_{\infty} ^i= \vect{F}_{\infty}>0$.
Thus there is an $\hat{H}> r_1 $ such that $f^i(\vect{x}) \geq (\vect{F}_{\infty} - \varepsilon) \norm{\vect{x}}$ for $ \vect{x} \in \mathbb{R}_+^n$ and $\norm{\vect{x}} \geq \hat{H}$.
Let $r_2=\max\{2r_1, \f{\hat{H}}{\s}\}$ and it follows that $\sum_{j=1}^nx_j(t) \geq \s \norm{\vect{x}}   \geq \hat{H}$
for $\vect{x}=(x_1,\dots,x_n) \in \partial\O_{r_2}$ and $t \in [0, \o]$.
Thus $f^i(\vect{x}(t)) \geq  (\vect{F}_{\infty} - \varepsilon)\sum_{j=1}^n x_j(t)$ for $\vect{x} \in \partial\O_{r_2}$ and $t \in [0, \o]$.
In view of Lemma ~\ref{f_estimate_>*}, we have
$$ \norm{\vect{T}_{\l}\vect{x}}   \geq  \l (\vect{F}_{\infty} - \varepsilon) \s \G \norm{\vect{x}}   >  \norm{\vect{x}}   \quad \textrm{for} \quad  \vect{x} \in \partial\O_{r_2}.$$
It follows from Lemma ~\ref{lm1} that
$$
i(\vect{T}_{\l}, \O_{r_1}, K)=1, \quad i(\vect{T}_{\l}, \O_{r_2}, K)=0.
$$
Thus
$i(\vect{T}_{\l}, \O_{r_2} \setminus \bar{\O}_{r_1}, K)=-1$.
Hence, $\vect{T}_{\l}$ has a fixed point in  $\O_{r_2} \setminus \bar{\O}_{r_1}$. Consequently, (\ref{eq1})
has a positive $\o$-periodic solution.

If $\vect{F}_{\infty}<\vect{F}_0$, then ${1}/{(\s\G\vect{F}_{0})} < \l < {1}/{(\chi \vect{F}_{\infty})}$.
It is easy to see that there
exists an $0<\varepsilon <\vect{F}_{0}$ such that
$$
\f{1}{ \s \G (\vect{F}_{0}- \varepsilon)} < \l < \f{1}{\chi (\vect{F}_{\infty}+\varepsilon) }.
$$
Now, turning to $\vect{F}_0$ and $\vect{F}_{\infty}$. Again,  there are a fixed index $i$ and  an $r_1>0$ such that $f^i_{0}=\vect{F}_{0}$ and
$f^i(\vect{x}) \geq (\vect{F}_0 - \varepsilon) \norm{\vect{x}}$ for $ \vect{x} \in \mathbb{R}_+^n$ and $\norm{\vect{x}} \leq r_1$.
Thus $f^i(\vect{x}(t)) \geq  (\vect{F}_0 - \varepsilon)\sum_{j=1}^n x_j(t)$ for $\vect{x}=(x_1,\dots,x_n) \in \partial\O_{r_1}$ and $t \in [0, \o]$.
We have by Lemma ~\ref{f_estimate_>*} that
$$
 \norm{\vect{T}_{\l}\vect{x}}  \geq \l(\vect{F}_0 - \varepsilon) \s \G \norm{\vect{x}}   > \norm{\vect{x}}   \quad \textrm{for} \quad  \vect{x} \in \partial\O_{r_1}.
$$
On the other hand, there is an $\hat{H}>r_1$ such that $f^i(\vect{x}) \leq (\vect{F}_{\infty} + \varepsilon) \norm{\vect{x}}$, $i=1,\dots,n,$ for $ \vect{x} \in \mathbb{R}_+^n$ and $\norm{\vect{x}} \geq \hat{H}$.
Let $r_2=\max\{2r_1, {\hat{H}}/{\s}\}$ and it follows that $\sum_{j=1}^n x_j(t) \geq \s \norm{\vect{x}}  \geq \hat{H}$
for $\vect{x}=(x_1,\dots,x_n) \in \partial\O_{r_2}$ and $t \in [0, \o]$.
Thus $f^i(\vect{x}(t)) \leq  (\vect{F}_{\infty} + \varepsilon)\sum_{j=1}^n x_j(t)$,  $i=1,\dots,n,$ for $\vect{x}=(x_1,\dots,x_n) \in \partial\O_{r_2}$ and $t \in [0, \o]$.
In view of Lemma ~\ref{f_estimate_<*}, we have
$$ \norm{\vect{T}_{\l}\vect{x}}   \leq  \l (\vect{F}_{\infty} + \varepsilon)  \chi \norm{\vect{x}}   <  \norm{\vect{x}}   \quad \textrm{for} \quad  \vect{x} \in \partial\O_{r_2}.$$
It follows from Lemma ~\ref{lm1} that
$$
i(\vect{T}_{\l}, \O_{r_1}, K)=0, \quad i(\vect{T}_{\l}, \O_{r_2}, K)=1.
$$
Thus
$i(\vect{T}_{\l}, \O_{r_2} \setminus \bar{\O}_{r_1}, K)=1$.
Hence, $\vect{T}_{\l}$ has a fixed point in  $\O_{r_2} \setminus \bar{\O}_{r_1}$. Consequently, (\ref{eq1})
has a positive $\o$-periodic solution. The proof is complete.

\subsection{Proof of asymptotic stability}\label{proof stability}
In the following, we present the details of the proof of the sufficient and necessary criteria for the the asymptotic stability of \eqref{SYSTEM}. The approach involves the contraction mapping principle and contradiction arguments.

\subsubsection{Proof of Theorem \ref{STH} (sufficient criteria)}

For any $t_0\geq0$, one can find a $\delta$ with $0\leq\delta < L$ such that
   $$\delta K+\alpha L\leq L,$$
where
$$K:=\max \{K_1,K_2,\dots,K_n\}, ~~K_i:=\sup_{t\geq t_0}\ld\exp\ld-\il^t_{t_0} a_i(s)ds\rd\rd.
$$
For any $\phi\in \mathscr{C}$, by ($H_2$), there is a unique solution of  (\ref{SYSTEM})
 $$\vect{x}(t)=\vect{x}(t,t_0,\phi), ~~\vect{x}_{t_0}=\phi,~~t\geq t_0.$$

We first show that the zero solution of \eqref{SYSTEM} is attractive, i.e.,
$$\lim_{t\rightarrow +\infty}\vect{x}(t,t_0,\phi)=0, ~~\phi\in \mathscr{C}_\delta :=\{\vect{x}\in \mathscr{C}:\norm{\vect{x}}\leq\delta \}.$$

Consider the initial value problem
\begin{equation}\label{eq7}
\vect{x}'(t)=-\vect{A}(t)\vect{x}(t) + \vect{B}(t)\vect{F}(\vect{x}_t), \ \ \vect{x}_{t_0}=\phi.
\end{equation}
Then the unique solution of (\ref{eq7}) satisfies
$$\vect{x}(t)=\vect{\Phi}(t,t_0)\phi(0)+ \il^t_{t_0}\vect{\Phi}(t,s)\vect{B}(s)\vect{F}(\vect{x}_s)ds,$$
where
$$\vect{\Phi}(t,s)=diag\left[\exp\ld-\il^t_s a_1(u)du\rd,\dots,\exp\ld-\il^t_s a_n(u)du\rd\right].
$$
Define
$$S=\{\vect{x}\in C([t_0-r,+\infty), \mathbb{R}^n) \mid \vect{x}_{t_0}=\phi, \norm{\vect{x}_t}\leq L,  t\geq t_0, \lim_{t\rightarrow+\infty}\vect{x}(t)= 0 \}.$$
Then $S$ is a complete metric space with the metric
$$
\rho(\vect{x},\vect{y})=\sup_{t\geq t_0-r}\abs{\vect{x}(t)-\vect{y}(t)}.$$
Define the mapping $\vect{P}:S\rightarrow C([t_0-r,+\infty), \mathbb{R}^n)$ by
\begin{equation*}
\begin{split}
(\vect{P}\vect{x})(t)
&= \vect{\phi}(t-t_0), \ \ \ \ \ \ \ \ \ \ \ \ \ \ \ \ \ \ \ \ \  \ \ \ \ \ \ \ \ \ \ \  \  t_0-r\leq t\leq t_0,\\
(\vect{P}\vect{x})(t)
&= \vect{\Phi}(t,t_0)\phi(0)+ \il^t_{t_0}\vect{\Phi}(t,s)\vect{B}(s)\vect{F}(\vect{x}_s)ds, \ \ \ \ \ \ \  t\geq t_0.\\
\end{split}
\end{equation*}

We claim that $\vect{P}$ maps $S$ into itself, i.e. $\vect{P}:S\rightarrow S$.
In fact, it is clear that $\vect{P}\vect{x}$ is continuous for $\vect{x} \in S$ and $(\vect{P}\vect{x})(t)= \vect{\phi}(t-t_0)$ for $t_0-r\leq t\leq t_0$.
For any $t\geq t_0$,
\begin{equation*}
\begin{split}
\abs{(\vect{Px})(t)}
&\leq \abs{\vect{\Phi}(t,t_0)\phi(0)}+ \il^t_{t_0} \abs{\vect{\Phi}(t,s)\vect{B}(s)} \cdot \abs{\vect{F}(\vect{x}_s)}ds\\
&\leq \abs{\vect{\Phi}(t,t_0)\phi(0)}+\il^{t}_{t_0} K_L \abs{\vect{\Phi}(t,s)\vect{B}(s)}\cdot\norm{\vect{x}_s}ds\\
&\leq \abs{\vect{\Phi}(t,t_0)\phi(0)}+\il^{t}_{t_0} K_L \abs{\vect{\Phi}(t,s)\vect{B}(s)}\cdot L ds\\
&\leq\delta K+\alpha L\leq L. \\
\end{split}
\end{equation*}
Then, $\norm{(\vect{P}\vect{x})_t}=\sup_{-r \leq \theta\leq 0} \abs{(\vect{P}\vect{x}(t+\theta)} \leq L.$

We now show that $\vect{P}\vect{x}(t)\rightarrow 0$ when $t\rightarrow +\infty.$ By (H3), we have
 $$\lim_{t\rightarrow +\infty}\exp\ld-\il^t_{t_0}a_i(s)ds\rd=0, ~i=1,2,\dots,n.$$
Thus the first term on the right-hand side of $(\vect{Px})(t)$ tends to zero. In the following, we show that the second term on the right side of $(\vect{Px})(t)$ also tends to zero. The fact $\vect{x}\in S$ implies that $\abs{\vect{x}(t)}\leq L$ for $t\geq t_0$. For any given $\varepsilon>0$, there exists a $t_1>0$ such that $\norm{\vect{x}_t}<\varepsilon$ for all $t\geq t_1$.  By (H3),  there exists a $t_2$ with $t_2>t_1$ such that
 $$\exp\ld-\il_{t_1}^ta_i(u)du\rd<\f{\varepsilon}{\alpha L},~~t>t_2.$$
For any $t>t_2$, one has
\begin{equation*}
\begin{split}
\abs{ \il^t_{t_0}\vect{\Phi}(t,s)\vect{B}(s)\vect{F}(\vect{x}_s)ds}
&\leq \il^{t_1}_{t_0} \abs{\vect{\Phi}(t,s)\vect{B}(s)}\cdot \abs{\vect{F}( \vect{x}_s)}ds+ \il^t_{t_1}\abs{ \vect{\Phi}(t,s)\vect{B}(s)}\cdot \abs{\vect{F}(\vect{x}_s)}ds\\
&\leq  K_L\il^{t_1}_{t_0}\abs{\vect{\Phi}(t,s)\vect{B}(s)}\cdot \norm{\vect{x}_s}ds+ K_L \il^t_{t_1}\abs{\vect{\Phi}(t,s)\vect{B}(s)}\cdot\norm{\vect{x}_s}ds\\
&\leq  K_L \il^{t_1}_{t_0}\abs{\vect{\Phi}(t,s)\vect{B}(s)}ds\cdot L+ K_L \il^t_{t_1}\abs{\vect{\Phi}(t,s)\vect{B}(s)}ds\cdot \varepsilon\\
&\leq \abs{\vect{\Phi}(t,t_1)}\cdot K_L \il^{t_1}_{t_0}\abs{\vect{\Phi}(t_1,s)\vect{B}(s)}ds\cdot L+\alpha\varepsilon\\
&\leq \abs{\vect{\Phi}(t,t_1)}\cdot \alpha \cdot L+\alpha\varepsilon\\
&\leq \varepsilon+\alpha\varepsilon=(1+\alpha)\varepsilon
\end{split}
\end{equation*}
Then $\lim_{t\rightarrow\infty}(\vect{P}\vect{x})(t)=0,$
and hence $(\vect{P}\vect{x}) \in S$.

We next claim that the map $\vect{P}$ is a contraction. For any $\vect{x}, \vect{y} \in S$,
$$\vect{Px}(t)=\vect{Py}(t),~~t_0-r\leq t\leq t_0.$$
For any $t\geq t_0$, one has
\begin{equation*}
\begin{split}
\abs{\vect{Px}(t)-\vect{Py}(t)}
&=\abs{\il^t_{t_0}\vect{\Phi}(t,s)\vect{B}(s)\vect{F}(\vect{x}_s)ds-
\il^t_{t_0}\vect{\Phi}(t,s)\vect{B}(s)\vect{F}( \vect{y}_s))ds}\\
&\leq \il^t_{t_0}\abs{ \vect{\Phi}(t,s)\vect{B}(s)}\cdot \abs{\vect{F}(\vect{x}_s)-
\vect{F}(\vect{y}_s)}ds\\
&\leq K_L\il^t_{t_0}\abs{ \vect{\Phi}(t,s)\vect{B}(s)}\cdot \norm{\vect{x}_s-\vect{y}_s}ds\\
&\leq K_L\il^t_{t_0}\abs{ \vect{\Phi}(t,s)\vect{B}(s)}\cdot \rho(\vect{x},\vect{y})ds\\
&\leq \alpha \rho(\vect{x},\vect{y}).\\
\end{split}
\end{equation*}
Then $\rho(\vect{P}\vect{x},\vect{P}\vect{y})=\sup\limits_{t\geq t_0} \abs{\vect{Px}(t)-\vect{Py}(t)}
\leq \alpha \rho(\vect{x},\vect{y}).$

Now, by the contraction mapping principle, $\vect{P}$ has a unique fixed point in $S$, which is the unique solution of (\ref{eq7}) and tends to zero as $t$ tends to
infinity. Obviously, the unique solution of (\ref{eq7}) is $\vect{x}(t)$. Therefore,
$$\lim_{t\rightarrow +\infty}\vect{x}(t,t_0,\phi)=0, \ \  \forall\phi\in \mathscr{C}_\delta.$$

To obtain the asymptotic stability, we need to show the zero solution
of (\ref{eq7}) is stable. For any given $0<\varepsilon< L$, we choose $0<\delta<\varepsilon$ such that $\delta K+\alpha \varepsilon<\varepsilon$. The solution of (\ref{eq7}) is $\vect{x}(t)=\vect{x}(t,t_0,\phi)$, $\vect{x}_{t_0}=\phi, \norm{\phi}\leq \delta$.  For $t\geq t_0$,
$$\vect{x}(t)=\vect{\Phi}(t,t_0)\phi(0)+ \il^t_{t_0}\vect{\Phi}(t,s)\vect{B}(s)\vect{F}(\vect{x}_s)ds.$$
We show that $\abs{\vect{x}(t)}< \varepsilon$ when $t\geq t_0$. In fact, $\abs{\vect{x}(t_0)}< \varepsilon,$ if there exists $t^\ast > t_0$ such that $\abs{\vect{x}(t^\ast)}=\varepsilon$ and $\abs{\vect{x}(s)}< \varepsilon$ when $t_0\leq s < t^\ast$, we have,
\begin{equation*}
\begin{split}
\abs{\vect{x}(t^\ast)}
&\leq \abs{\vect{\Phi}(t^\ast,t_0)\phi(0)}+ \il^{t^\ast}_{t_0}\abs{ \vect{\Phi}(t,s)\vect{B}(s)}\cdot \abs{\vect{F}(\vect{x}_s)}ds\\
&\leq \abs{\vect{\Phi}(t,t_0)\phi(0)}+\il^{t}_{t_0} K_L \abs{\vect{\Phi}(t,s)\vect{B}(s)}\cdot\norm{\vect{x_s}}ds\\
&\leq  K\delta+ \alpha \varepsilon < \varepsilon,\\
\end{split}
\end{equation*}
which is a contradiction.
Therefore, $\abs{\vect{x}(t)}< \varepsilon$ for all $t\geq t_0$.
The proof is complete.

\subsubsection{Proof of Theorem \ref{STH1} (Necessary criteria)}
The proof is based on the contradiction argument. Suppose that (H3) fails. By (H6), there exists a sequence$\{t_n\}$, $t_n\rightarrow +\infty$ when $n\rightarrow +\infty$ such that
  $$\lim_{n\rightarrow +\infty} \il_0^{t_n}a_i(s)ds=q$$
for some $q\in \mathbb{R}$. Then, one can choose a positive constant $Q>0$ such that
 $$-Q\leq \il_0^{t_n}a_i(s)ds\leq Q, \ \ \ n=1,2,\dots .$$
By (H5),
$$ K_L \il_0^{t_n}\exp\ld-\il_s^{t_n}a_i(u)du\rd\abs{b_i(s)}ds\leq\alpha,$$
then
\begin{equation*}
\begin{array}{l}
K_L \il_0^{t_n}\exp\ld\il_0^sa_i(u)du\rd\abs{b_i(s)}ds\\
=K_L \il_0^{t_n} \exp\ld\il_0^{t_n}a_i(u)du\rd\cdot\exp\ld-\il_s^{t_n}a_i(u)du\rd\abs{b_i(s)}ds\\
 \leq\alpha \exp\ld\il_0^{t_n}a_i(u)du\rd\leq e^Q.\\
\end{array}
\end{equation*}
Whence, the sequence $\ld K_L \il_0^{t_n}\exp\{\il_0^sa_i(u)du\}\abs{b_i(s)}ds\rd$ is bounded and there exists
a convergent subsequence. For brevity in notation, we still assume the
sequence $K_L \il_0^{t_n}\exp\ld\il_0^sa_i(u)du\rd\abs{b_i(s)}ds$ is convergent, say its limit is $\gamma \in \mathbb{R}_+$. We can choose a positive integer $m$ so large that
 $$ K_L \il_{t_m}^{t_n}\exp\ld\il_0^sa_i(u)du\rd\abs{b_i(s)}ds\leq \f{1-\alpha}{2K^2 e^{2Q}},~~n\geq m.$$
Consider the solution $\vect{x}(t)=\vect{x}(t,t_m,\phi)$ with $\phi=(\phi_1, \cdots, \phi_n)^T$ and $\phi_i(s-t_m)\equiv \delta$ for $s \in [t_m-r, t_m]$ and $\abs{\vect{x}(t)}\leq L$ for all $t\geq t_m$. For $t\geq t_m$, one has
\begin{equation*}
\begin{split}
\abs{\vect{x}(t)}
&\leq \abs{\vect{\Phi}(t,t_0)\phi(0)}+ \il^t_{t_m} \abs{\vect{\Phi}(t,s)\vect{B}(s)}\cdot \abs{\vect{F}(\vect{x}_s)}ds\\
&\leq K\delta+\il^{t}_{t_m} K_L\abs{\vect{\Phi}(t,s)\vect{B}(s)}\cdot \norm{\vect{x_s}}ds\\
&\leq K\delta+\alpha \norm{\vect{x_t}}.
\end{split}
\end{equation*}
Then
$$\norm{\vect{x}_t}\leq \f{K\delta}{1-\alpha},~~t\geq t_m.$$
On the other hand, for $n\geq m$ large enough, we also have
$$\vect{x}(t_n)=\vect{x}(t_m)\vect{\Phi}(t_n,t_m)+ \il^{t_n}_{t_m}\vect{\Phi}(t_n,s)\vect{B}(s)\vect{F}(\vect{x}_s)ds,$$
then
 \begin{equation*}
\begin{split}
\abs{\vect{x}(t_n)}
&\geq \delta \abs{\vect{\Phi}(t_n,t_m)}- \il^{t_n}_{t_m}\abs{\vect{\Phi}(t_n,s)\vect{B}(s)}\cdot \abs{\vect{F}(\vect{x}_s)}ds\\
&\geq \delta \abs{\vect{\Phi}(t_n,t_m)}- K_L \il^{t_n}_{t_m} \abs{\vect{\Phi}(t_n,s)\vect{B}(s)}\cdot \norm{\vect{x}_s}ds\\
&\geq \delta \abs{\vect{\Phi}(t_n,t_m)}-\f{K\delta}{1-\alpha} \abs{\vect{\Phi}(t_n,0)} K_L \il^{t_n}_{t_m} \abs{\vect{\Phi}(0,s)\vect{B}(s)}ds.
\end{split}
\end{equation*}
Since $\abs{\vect{\Phi}(t_n,0)}\leq K$ and
\begin{equation*}
\begin{split}
\exp\ld-\il^{t_n}_{t_m}a_i(u)du\rd
&=\exp\ld-\il^{t_n}_0a_i(u)du\rd\cdot\exp\ld\il^{t_m}_0a_i(u)du\rd\geq \exp\{-2Q\},
\end{split}
\end{equation*}
then
$$\abs{\vect{x}(t_n)}\geq \delta \exp\{-2Q\}-\f{K\delta}{1-\alpha}  \cdot K  \cdot \f{(1-\alpha) \exp\{-2Q\}}{2K^2}=\f{1}{2}\delta \exp\{-2Q\}.$$
Hence, $\lim_{t\rightarrow+\infty} \vect{x}(t)\neq 0$. One has a contradiction. The proof is complete.

\newpage

\begin{table}[htp]\caption{Positive periodic solutions for system (\ref{eq1})}\label{tabledef}
\footnotesize
\begin{center}
\begin{tabular}{|c|c|c|c|}\hline
 $i_0$ & $i_{\infty}$ & $\l$ &  Number of periodic solutions \\\hline
 & & $ 0< \l < \f{1}{c_2 \chi}$ or $\lambda>\f{1}{\s \G c_1}$ & 0\\ \cline{3-4}
 & \raisebox{2ex}[0pt]{0} & $\f{1}{ \s \G} \f{1}{\max\{\vect{F}_0, \vect{F}_{\infty}\}} < \l < \f{1}{\chi} \f{1}{\min \{\vect{F}_0, \vect{F}_{\infty} \}}$ & 1\\\cline{2-4}
  &  & $0<\l <\f{1}{M(1)\chi}$ & 1\\\cline{3-4}
\raisebox{2ex}[0pt]{0} & \raisebox{2ex}[0pt]{1} & $\lambda>\f{1}{\s \G c_1}$ & 0 \\\cline{2-4}
  &  & $0<\l <\f{1}{M(1)\chi}$    & 2\\\cline{3-4}
  & \raisebox{2ex}[0pt]{2} & $\lambda>\f{1}{\s \G c_1}$ & 0 \\\cline{1-4}
  &   & $0< \l < \f{1}{c_2 \chi}$ & 0 \\\cline{3-4}
1 & \raisebox{2ex}[0pt]{0} & $\l>\f{1}{m(1)\G}$ & 1 \\\cline{2-4}
  & 1 &
 $\lambda>0$
 & 1 \\\cline{1-4}
  &   & $0< \l < \f{1}{c_2 \chi}$ & 0 \\\cline{3-4}
\raisebox{2ex}[0pt]{2} & \raisebox{2ex}[0pt]{0} & $\l>\f{1}{m(1)\G}$ & 2 \\\cline{1-4}
\end{tabular}\\
\label{existence}
\end{center}
\end{table}

\begin{figure}[htp]
\small
\centering
\includegraphics[height=2in,width=5in]{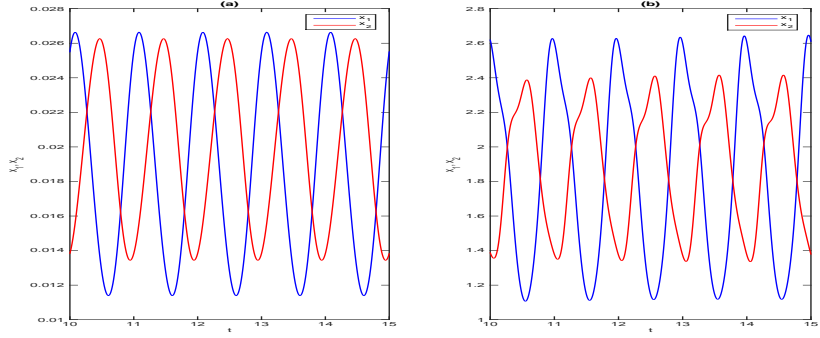}
\caption{\footnotesize  {(\ref{eqexam}) admits positive periodic solutions.
$(a) ~\lambda=0.1.$
$(b) ~\lambda=401.$
Here $\tau=0.1$.}}
\label{fkwy21fig}
\end{figure}

\begin{figure}[htp]
\small
\centering
\includegraphics[height=2in,width=5in]{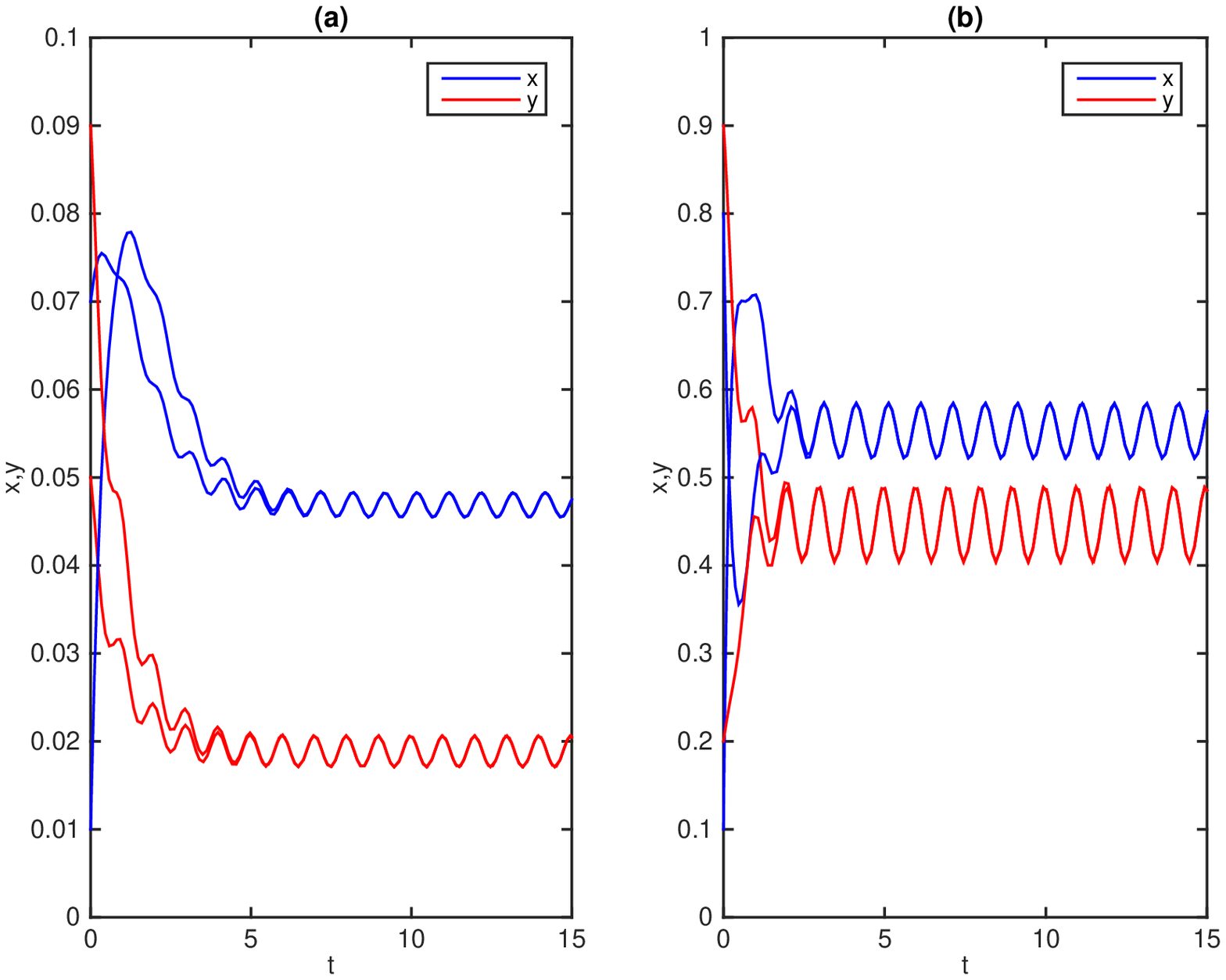}
\caption{ \footnotesize {(\ref{Model1}) with $f(y)=y^2/(\theta^2+y^2)$ admits a positive $1$-periodic solution being asymptotically stable. (a) The solutions of \eqref{Model1} starting at $(x(0),y(0)) = (0.07,0.05)$ and $(x(0),y(0)) = (0.01, 0.09)$ tend to the positive periodic solution. (b) The solutions of \eqref{Model1} starting at $(x(0),y(0)) = (0.1,0.9)$ and $(x(0),y(0)) = (0.8, 0.2)$ tend to the positive periodic solution. The values of parameters are same as those in Fig. \ref{figexp2}.}}
\label{figexp12pic}
\end{figure}

\begin{figure}[htp]
\small
\centering
\includegraphics[height=2in,width=5in]{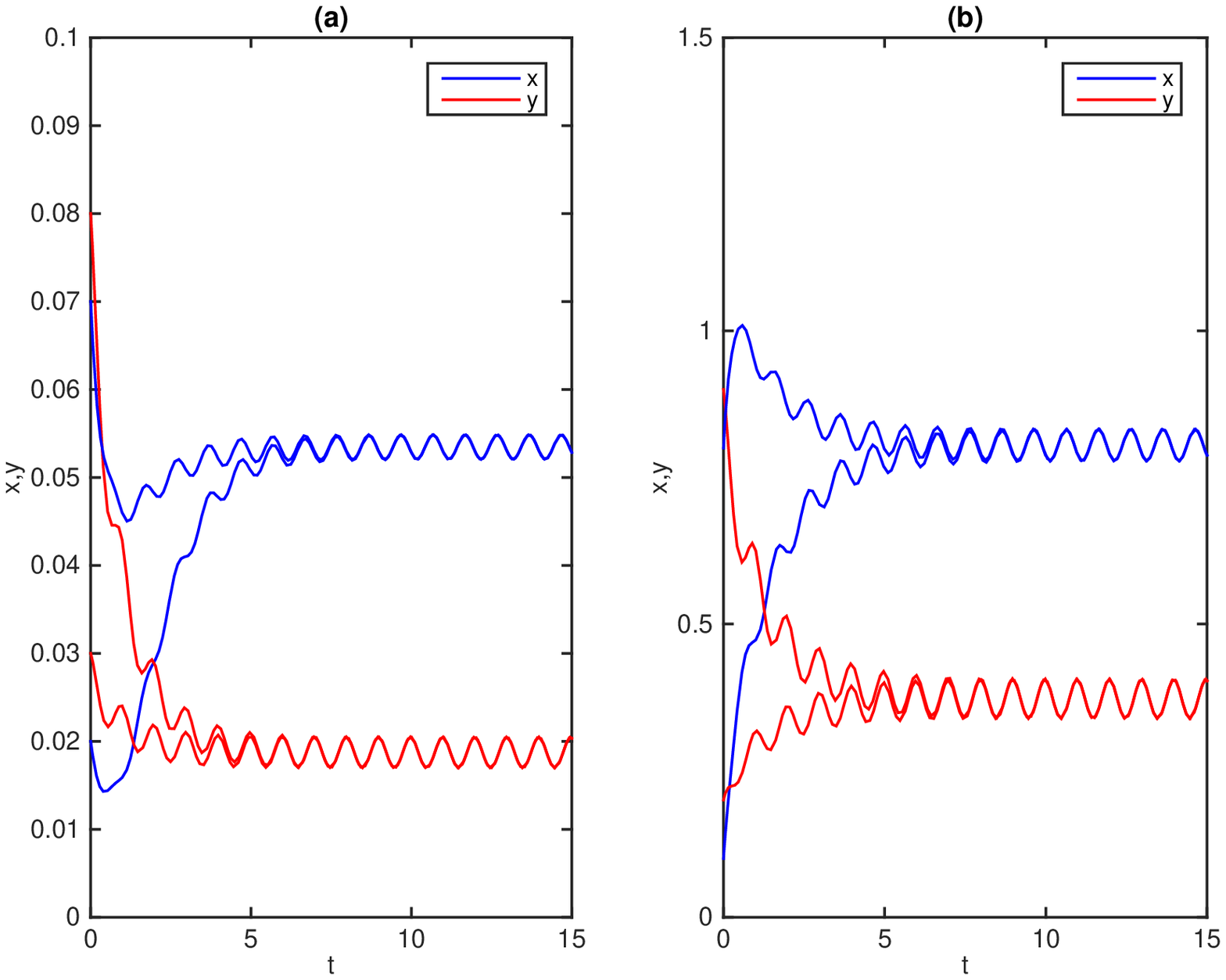}
\caption{ \footnotesize {(\ref{Model1}) with $f(y)=\theta^2/(\theta^2+y^2)$ admits a positive $1$-periodic solution being asymptotically stable.
(a) The solutions of \eqref{Model1} starting at $(x(0),y(0)) = (0.02,0.08)$ and $(x(0),y(0)) = (0.07, 0.03)$ tend to the positive periodic solution. (b) The solutions of \eqref{Model1} starting at $(x(0),y(0)) = (0.1,0.9)$ and $(x(0),y(0)) = (0.8, 0.2)$ tend to the positive periodic solution.
The values of parameters are same as those in Fig. \ref{figexp2}.}}
\label{figexp3}
\end{figure}

\begin{figure}[htp]
\small
\centering
\includegraphics[height=2in,width=5in]{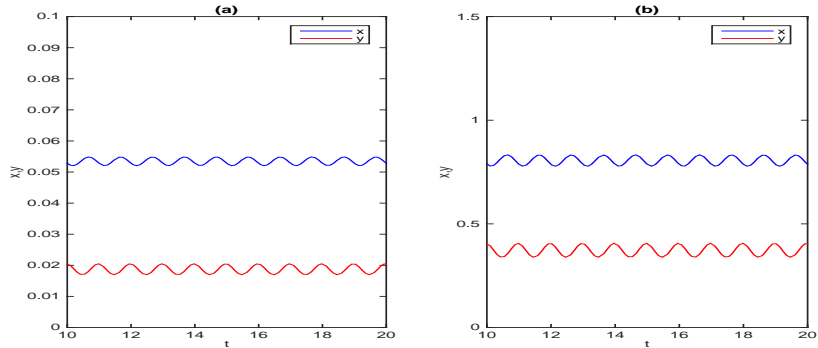}
\caption{ \footnotesize {\eqref{Model1} with $f(y)=\theta^2/(\theta^2+y^2)$ admits positive periodic solution. (a) $a=0.2, b=2, c=0.02, \theta=0.02$. (b) $a=5, b=4, c=1.2, \theta=0.5.$ Here $h(t)=1+0.6sin(2\pi t)$.}}
\label{figexp2}
\end{figure}
\begin{figure}[htp]
\small
\centering
\includegraphics[height=2in,width=5in]{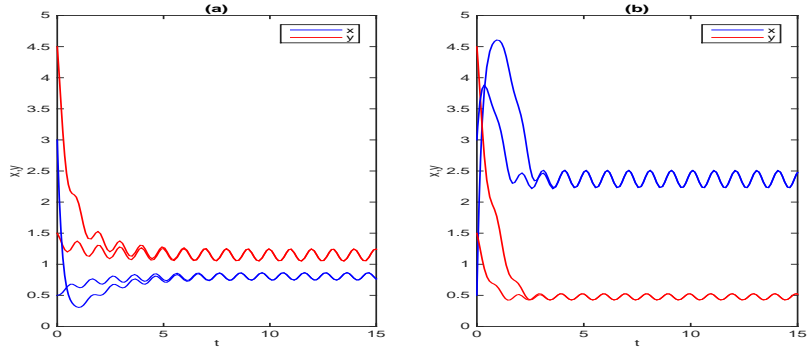}
\caption{ \footnotesize {\eqref{Model1} admits positive periodic solution being asymptotically stable for large $a$ and $c$, where the solutions of \eqref{Model1} starting at $(x(0),y(0)) = (0.5,4.5)$ and $(x(0),y(0)) = (3, 1.5)$ tend to the positive periodic solution. Here $e(t)=1+0.6sin(2\pi t)$, $a=20, b=4, c=8, \theta=0.5$. (a) $f(y)=\theta^2/(\theta^2+y^2)$; (b)  $f(y)=y^2/(\theta^2+y^2)$. }}
\label{figexp7}
\end{figure}

\begin{figure}[htp]
\small
\centering
\includegraphics[height=2in,width=5in]{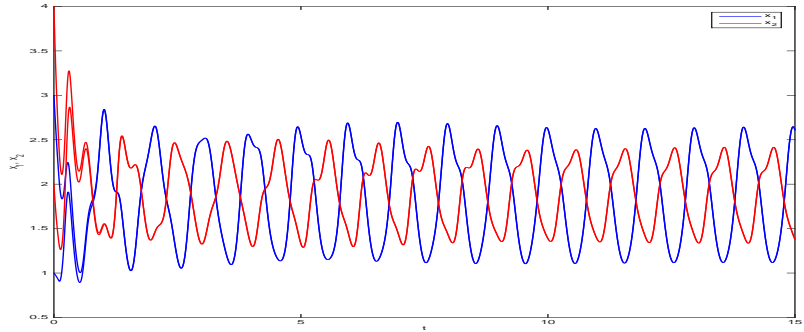}
\caption{\footnotesize  {\eqref{eqexam} has a positive periodic solution being asymptotically stable for large $\lambda$ not satisfying \eqref{conditioneq}. The solutions of \eqref{eqexam} starting at $(x_1(\theta), x_2(\theta))=(1, 4)$ and $(x_1(\theta), x_2(\theta))=(3, 2)$ for $\theta \in[-0.1,0]$ tends to the positive periodic solution. Here $\lambda=401, \tau=0.1$.}}
\label{fkwy23fig}
\end{figure}
\begin{figure}[htp]
\small
\centering
\includegraphics[height=2in,width=5in]{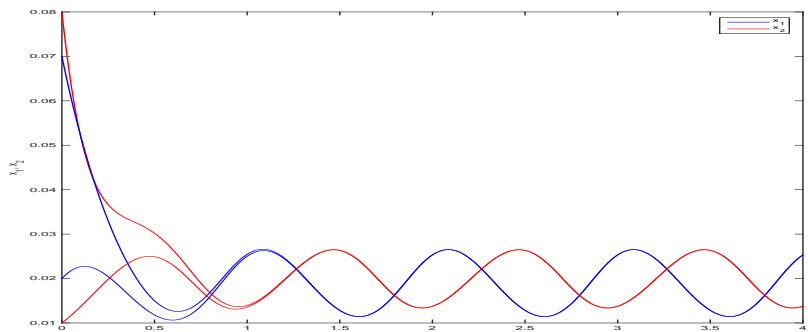}
\caption{\footnotesize {\eqref{eqexam} has a positive periodic solution being asymptotically stable for small $\lambda$ satisfying \eqref{conditioneq}. The solutions of \eqref{eqexam} starting at $(x_1(\theta), x_2(\theta))=(0.02, 0.08)$ and $(x_1(\theta), x_2(\theta))=(0.07, 0.01)$ for $\theta \in[-0.1,0]$ tends to the positive periodic solution. Here $\lambda=0.1, \tau=0.1$.}}
\label{fkwy22fig}
\end{figure}
\begin{figure}[htp]
\small
\centering
\includegraphics[height=3in,width=5in]{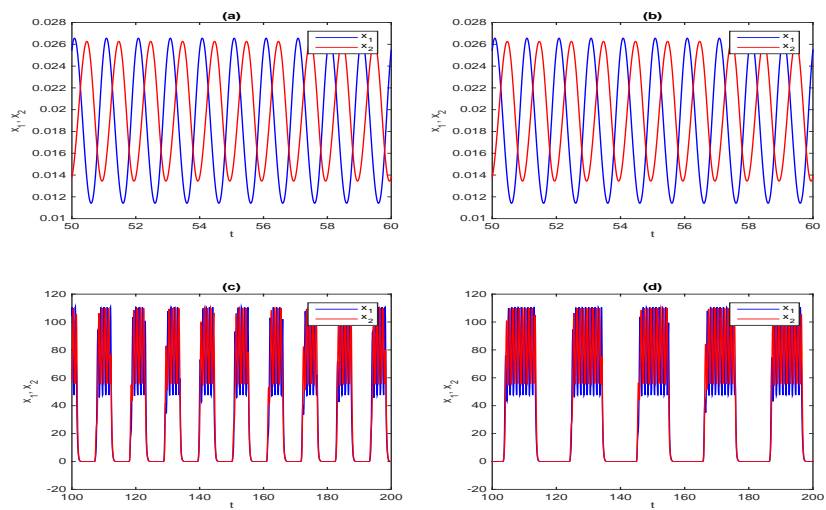}
\caption{\footnotesize {Effect of delay $\tau$ on the positive periodic solution of \eqref{eqexam}.   (a)  $\lambda=0.1, \tau=5.$ (b)$\lambda=0.1$, $\tau=10$. (c) $\lambda=401$, $\tau=5$. (d) $\lambda=401$, $\tau=10$.}}
\label{delayeffectexmp}
\end{figure}

\end{document}